\documentclass[10pt]{amsart}
\usepackage{amsmath, amssymb, amsthm, amscd, graphicx}

\theoremstyle{plain}
\newtheorem{prop}{Proposition}[section]
\newtheorem*{thrm*}{Theorem}
\newtheorem{claim}{Claim}
\newtheorem{coro}[prop]{Corollary}
\newtheorem{lemm}[prop]{Lemma}
\newtheorem{conj}[prop]{Conjecture}

\theoremstyle{definition}
\newtheorem{defi}[prop]{Definition}
\newtheorem{nota}[prop]{Notation}
\newtheorem{exam}[prop]{Example}
\newtheorem{rema}[prop]{Remark}
\newtheorem*{rema*}{Remark}

\numberwithin{equation}{section}

\renewcommand\a{\alpha}
\renewcommand\aa{a}%

\newcommand\AAA[1]{A_{#1}}
\newcommand\AAAA{{\boldsymbol{A}}}
\newcommand\aab{\overline{a}}
\newcommand\aam{a-1}
\newcommand\aap{a+1}
\newcommand\Abs[1]{\big\vert#1\big\vert}%
\newcommand\act{\mathbin{\scriptscriptstyle{\bullet}}}%
\newcommand\add[2]{\mathrm{add}_{#1}(#2)}%
\newcommand\Add{\mathbb{A}}%
\newcommand\Ass[1]{K_{#1}}%

\renewcommand\b{\beta}
\newcommand\bb{b}%
\newcommand\bbb{\overline{b}}
\newcommand\bbm{b-1}
\newcommand\bbp{b+1}

\newcommand\card{\mathtt{\#}}
\newcommand\Card[2]{\mathtt{\#}_{_{\!\scriptstyle#1}}\!#2}

\newcommand\cc{c}
\newcommand\CC{C}
\newcommand\ccb{\overline{c}}

\newcommand\ccp{c+1}
\newcommand\cha[1]{w_{#1}}
\newcommand\coll[1]{\mathrm{coll}_{#1}}
\newcommand\cov[1]{\vartriangleleft_{#1}}
\newcommand\cove[1]{\trianglelefteq_{#1}}
\newcommand\ccov[1]{\mathrel{\vartriangleright^{\!*}_{#1}}}
\newcommand\ccove[1]{\mathrel{\trianglerighteq^{\!*}_{#1}}}

\renewcommand{\d}{\delta}%
\newcommand\dd{d}
\newcommand\ddb{\overline{d}}

\newcommand\diam{d}
\newcommand\dist{\mathrm{dist}}
\newcommand\Dist[1]{#1\text{-}\mathrm{dist}}

\newcommand\dom[1]{\tau(#1)}

\newcommand{\e}{\varepsilon}
\newcommand{\ea}{{\scriptstyle\varnothing}}%
\newcommand\ee{e}

\newcommand\eep{e+1}
\newcommand\ellb{\overline{\ell}}
\newcommand\et{\emptyset}

\newcommand\ff{f}
\newcommand\ffm{f-1}

\newcommand\g{\gamma}
\newcommand\gcp{\mathbin{\scriptstyle\wedge}}
\let\ge=\geqslant
\renewcommand\gg{g}

\newcommand\gse{\mathord{\geqslant}}%
\newcommand\gs{\mathord{>}}%

\newcommand\hh{h}
\newcommand\hR{\mathrm{h}_{\scriptscriptstyle R}}

\newcommand\ie{\textit{i.e.}}
\newcommand\ii{i}
\newcommand\II{I}
\newcommand\iim{{i-1}}
\newcommand\iip{{i+1}}
\newcommand\ince{\subseteq}
\newcommand\inv{^{-1}}

\newcommand\jj{j}
\newcommand\JJ{J}
\newcommand\jjm{{j-1}}

\newcommand\kk{k}
\newcommand\KK{K}

\newcommand\kkp{{k+1}}

\newcommand\Lab[1]{\mathrm{Lab}(#1)}
\let\le=\leqslant
\newcommand\Ldots{...\,}

\renewcommand{\lg}[1]{\vert#1\vert}
\newcommand\ls{\mathord{<}}%
\newcommand\lse{\mathord{\leqslant}}%

\newcommand\minus{\!-\!}
\newcommand\mm{m}

\newcommand\mmm{m-1}
\newcommand\mmp{m+1}

\newcommand\name[2]{\nu(#1,#2)}%
\newcommand\Name[3]{\nu_{#1}(#2,#3)}%
\newcommand\Nat{\mathbb{N}}
\newcommand\nn{n}
\newcommand\NN{N}
\newcommand\nnb{{\overline\nn}}
\newcommand\nnm{{\nn-1}}
\newcommand{\nno}{{n-1}}
\newcommand\nnp{{n+1}}

\newcommand\op{\mathord{{}^{\scriptstyle\wedge}}}

\newcommand\Part[3]{[\![#2,#3]\!]_{#1}}

\newcommand\plus{\!+\!}
\newcommand\pp{p}
\newcommand\ppm{{p-1}}

\newcommand\ppp{{p+1}}
\newcommand\pppp{{p+2}}

\newcommand\prefe{\mathrel{\raisebox{0.2ex}{$\scriptstyle\sqsubseteq$}}}

\newcommand\qq{q}
\def\Qm(#1, #2, #3, #4){(\,#1\,, \,#2\,, \,#3\,, \,#4\,)^-}
\def\Qp(#1, #2, #3, #4){(\,#1\,, \,#2\,, \,#3\,, \,#4\,)^+}
\def\Qpm(#1, #2, #3, #4){(\,#1\,, \,#2\,, \,#3\,, \,#4\,)^\pm}
\newcommand\qqm{{q-1}}

\newcommand\qqp{{q+1}}
\newcommand\qqpp{{q+2}}

\def\resp{\emph{resp.}~}
\newcommand\rr{r}

\newcommand\rrp{{r+1}}
\newcommand\RRR{\mathbb{R}}

\newcommand\sh{\mathrm{sh}}
\newcommand\size[1]{\vert#1\vert}

\def\Sp#1{\langle#1\rangle}
\renewcommand\ss{s}

\newcommand\ssp{{s+1}}
\newcommand\sub[2]{#2_{(#1)}}

\newcommand\tar[1]{\tau'(#1)}
\newcommand\Th[2]{\Phi(#1,#2)}
\newcommand\tI{\mathrm{I}}
\newcommand\tIm{\tI^-}
\newcommand\tIp{\tI^+}
\newcommand\tIpm{\tI^\pm}
\newcommand\tII{\mathrm{I\!I}}
\newcommand\tIIm{\tII^-}
\newcommand\tIIp{\tII^+}
\newcommand\tIIpm{\tII^\pm}
\newcommand\tIII{\mathrm{I\!I\!I}}
\newcommand\tIIIm{\tIII^-}
\newcommand\tIIIp{\tIII^+}
\newcommand\tIIIpm{\tIII^\pm}
\newcommand\tIV{\mathrm{I\!V\!}}

\newcommand\tIVp{\tIV^+}
\newcommand\tV{\mathrm{V\!}}
\newcommand\tVm{\tV^-}
\newcommand\tVp{\tV^+}
\newcommand\tVpm{\tV^\pm}
\newcommand\tVI{\mathrm{V\!I}}
\newcommand\tVIm{\tVI^-}
\newcommand\tVIp{\tVI^+}
\newcommand\tVIpm{\tVI^\pm}
\newcommand\TT{T}
\newcommand\TTb{\overline{\TT}}
\newcommand\TTs{T_*}
\newcommand\TTsb{\overline{\TT}_*}
\newcommand\TTT{T'}
\newcommand\TTTb{\overline{\TT}'}

\newcommand\ut{{\scriptstyle\bullet}}

\def\VR(#1,#2){\vrule width0pt height#1mm depth#2mm}
\newcommand\vv{v}

\newcommand\ww{w}

\newcommand\xx{x}

\newcommand\yy{y}

\newcommand\zz{z}

\begin{document}

\hfill{\tiny 2009-01}

\author{Patrick DEHORNOY}
\address{Laboratoire de Math\'ematiques Nicolas Oresme,
Universit\'e de Caen, 14032 Caen, France}
\email{dehornoy@math.unicaen.fr}
\urladdr{//www.math.unicaen.fr/\!\hbox{$\sim$}dehornoy}

\title{On the rotation distance between binary trees}



\begin{abstract}
We develop combinatorial methods for computing the rotation
distance between binary trees, \ie, equivalently, the flip distance between
triangulations of a polygon. As an application, we prove that, for each~$\nn$, there exist
size~$\nn$ trees at distance $2\nn - O(\sqrt{\nn})$.
\end{abstract}

\maketitle

If $\TT, \TTT$ are finite binary rooted trees, one says that $\TTT$ is
obtained from~$\TT$ by one rotation if $\TTT$ coincides with~$\TT$
except in the neighbourhood of some inner node where the branching
patterns respectively are 
$$\VR(6,5)
\begin{picture}(11,0)(0,4)
\put(0.5,0){\includegraphics{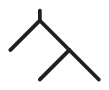}}
\put(4,9.2){.}
\put(4,10){.}
\put(4,10.8){.}
\put(0,3){.}
\put(-0.5,2.5){.}
\put(-1,2){.}
\put(4,-0.4){.}
\put(4,-1.2){.}
\put(4,-2){.}
\put(11,0){.}
\put(11.5,-0.5){.}
\put(12,-1){.}
\end{picture}
\mbox{\qquad and \qquad}
\begin{picture}(15,0)(0,4)
\put(0.5,0){\includegraphics{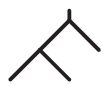}}
\put(7,9.2){.}
\put(7,10){.}
\put(7,10.8){.}
\put(0,0){.}
\put(-0.5,-0.5){.}
\put(-1,-1){.}
\put(7,-0.4){.}
\put(7,-1.2){.}
\put(7,-2){.}
\put(11,3){.}
\put(11.5,2.5){.}
\put(12,2){.}
\end{picture}.$$
Under the standard correspondence between trees and bracketed
expressions, a rotation corresponds to moving a pair of brackets
using the associativity law. If two trees~$\TT, \TTT$
have the same size (number of inner nodes), one can always
transform~$\TT$ to~$\TTT$ using finitely many rotations. The
\emph{rotation  distance} $\dist(\TT, \TTT)$ is the minimal number of
rotations needed in the transformation, and $\dd(\nn)$ will denote the
maximum of~$\dist(\TT, \TTT)$ for~$\TT, \TTT$ of size~$\nn$. Then
$\dd(\nn)$ is the diameter of the
$\nn$th associahedron, the graph~$\KK_\nn$ whose vertices are
size~$\nn$ trees and where $\TT$ and~$\TTT$ are adjacent if and only if
$\dist(\TT, \TTT)$ is one.

There exists a one-to-one correspondence between  size~$\nn$ trees
and triangulations of an $(\nn+2)$-gon. Under this
correspondence, a rotation in a tree translates into a \emph{flip} of the
associated triangulation, \ie, the operation of exchanging
diagonals in the quadrilateral formed by two adjacent triangles. So
$\dd(\nn)$ is also the maximal flip distance between two triangulations of
an $(\nn+2)$-gon.

In~\cite{STT}, using a simple counting argument, D.\,Sleator, R.\,Tarjan,
and W.\,Thurston prove the inequality $\dd(\nn) \le 2\nn-6$ for $\nn \ge
11$ and, using an argument of hyperbolic geometry, they prove $\dd(\nn)
\ge 2\nn-6$ for $\nn \ge \NN$, where $\NN$ is some ineffective
(large) integer. A brute force argument gives $\dd(\nn)
= 2\nn-6$ for $11 \le \nn \le 19$. It is natural to
conjecture $\dd(\nn) = 2\nn - 6$ for $\nn \ge 11$,
and to predict the existence of a combinatorial proof.
After~\cite{STT}, various related questions have been
addressed~\cite{Luc, Rog, HOS, HuN}, or~\cite{BoH} for a general
survey, but it seems that no real progress has been made on the above
conjecture. 

The aim of this paper is to develop combinatorial methods for addressing
the problem and, more specifically, for proving lower bounds on the
rotation distance between two trees. At the moment, we have no complete
determination of the value of~$\dd(\nn)$, but we establish a lower bound
in~$2\nn - O(\sqrt\nn)$ that is valid for each~$\nn$.

\begin{thrm*}
For $\nn = 2\mm^2$, we have $\dd(\nn) \ge 2\nn - 2\sqrt{2\nn} + 1$.
\end{thrm*}

We shall develop two approaches, which correspond to two different ways
of specifying a rotation in a tree. The first method takes the
\emph{position} of the subtree that is rotated into account. This viewpoint
naturally leads to introducing a partial action of Thompson's group~$F$
on trees and to expressing the rotation distance between two trees~$\TT,
\TTT$ as the length of the element of~$F$ that maps~$\TT$ to~$\TTT$
with respect to a certain family of generators. This approach is very natural
and it easily leads to a lower bound in $\frac32\nn +
O(1)$ for~$\dd(\nn)$. However, due to the lack of control on the
geometry of the group~$F$, it seems difficult to obtain higher lower
bounds in this way.

The second approach takes \emph{names}, rather than positions, into
account: names are given to the leaves of the trees, and one specifies a
rotation using the names of certain leaves that characterize the considered
rotation. This approach leads to partitioning the
associahedron~$\KK_\nn$ into regions separated by sort of
discriminant curves. Then, one proves that two trees~$\TT,
\TTT$ are at distance at least~$\ell$ by showing that any path
from~$\TT$ to~$\TTT$ through~$\KK_\nn$ necessarily intersects at
least~$\ell$ pairwise distinct discriminant curves. Progressively refining
the approach finally leads to the lower bound $2\nn - O(\sqrt\nn)$. No
obstruction a priori forbids to continue up to $2\nn-6$ but a few more
technical ingredients will probably be needed.

The paper is organized as follows. After setting the framework in
Section~\ref{S:Rotations}, we develop the approach based on
positions in Section~\ref{S:Positions}, and use it to deduce lower bounds
for~$\dd(\nn)$ that lie in~$\nn + O(1)$, and then in $\frac32\nn +
O(1)$. Section~\ref{S:Names} presents the approach based on names,
introducing the so-called covering relation, a convenient
way of describing the shape of a tree in terms of the names attributed to its
leaves. This leads to a new proof for a lower bound in $\frac32\nn + O(1)$.
In Section~\ref{S:Collapsing}, we introduce collapsing, an operation that
consists in erasing some leaves in a tree, and use it to improve the
previous bound to $\frac53\nn + O(1)$. Finally, in Section~\ref{S:Sqrt},
applying the same method in a more tricky way, we establish the 
$2\nn - O(\sqrt\nn)$ lower bound.

We use~$\Nat$ for the set of all nonnegative integers.

\section{Trees, rotations, and triangulations}
\label{S:Rotations}

\subsection{Trees}
\label{S:Trees}

All trees we consider are finite, binary, rooted, and ordered (for each
inner node: the associated left and right subtrees are identified). We
denote by~$\ut$, the tree consisting of a single vertex. If $\TT_0,
\TT_1$ are trees, $\TT_0 \op \TT_1$ is the tree whose left subtree
is~$\TT_0$ and right subtree is~$\TT_1$.  The \emph{size}~$\size\TT$
of a tree~$\TT$ is the number of symbols~$\op$ in the (unique)
expression of~$\TT$ in terms of~$\ut$ and~$\op$. Thus $\ut \op ((\ut
\op \ut) \op \ut)$ is a typical tree, usually displayed  as 
\begin{equation}
\label{E:Tree}
\VR(6,2)\begin{picture}(11,0)(0,4)
\put(0,0){\includegraphics{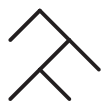}}
\end{picture}. 
\end{equation}
Its decomposition comprises three carets, so its size is~$3$.

Certain special trees will play a significant role, namely those such that, for
each inner node, only one of the associated subtrees may have a
positive size. Such a tree is completely determined by a
sequence of~$0$'s and~$1$'s, called its \emph{spine}.

\begin{defi} (See Figure~\ref{F:Spine}.)
For~$\a$ a finite sequence of~$0$'s and~$1$'s, the \emph{thin tree} with \emph{spine}~$\a$,
denoted~$\Sp\a$, is recursively defined by the rules
\begin{equation}
\label{E:Spine}
\Sp{\ea} = \bullet, \
\Sp0 = \Sp1 = \ut\op\ut, \
\text{and}\quad
\Sp{0\a} = \Sp\a \op \ut, \ 
\Sp{1\a} = \ut\op\Sp\a 
\text{\ for $\a \not= \ea$}.
\end{equation}
\end{defi}

For instance, the tree of~\eqref{E:Tree} is thin, with spines~$100$
and~$101$. Defining the spine so that it is not unique may appear
surprising, but, in this way,  $\Sp\a$ has size~$\nn$ for $\a$ of
length~$\nn$, which will make statements simpler.

\begin{figure}[htb]
\begin{picture}(100,32)(-6,0)
\put(0,-0.5){\includegraphics{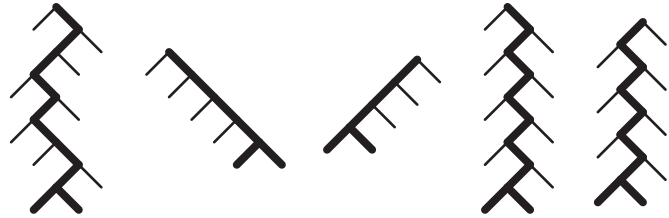}}
\end{picture}
\caption{\sf\smaller Thin trees, with their spines in bold: from left to
right,  $\Sp{100101100}$, which is also $\Sp{100101101}$, the
right comb of size~$5$, the left comb of size~$4$, the right zigzag of
size~$9$, the left zigzag of size~$8$.}
\label{F:Spine}
\end{figure}

Some particular families of thin trees will often appear in the sequel,
namely

- \emph{right combs}~$\Sp{111...}$ and their counterparts \emph{left
combs}~$\Sp{000...}$,

- \emph{right zigzags}~$\Sp{101010...}$ and \emph{left 
zigzags}~$\Sp{010101...}$---see Figure~\ref{F:Spine}.

\subsection{Rotations}
\label{S:Rotating}

For each vertex~$\vv$ (inner node or leaf) in a tree~$\TT$, there
exists a unique subtree of~$\TT$ with root in~$\vv$---see
\eqref{E:Subtree} below for a more formal definition. This subtree
will be called the \emph{$\vv$-subtree} of~$\TT$. For instance, if $\vv$ is
the root of~$\TT$, then the $\vv$-subtree of~$\TT$ is $\TT$~itself. Then
we can define the rotations mentioned in the introduction as
follows.

\begin{defi}
\label{D:Rotation}
(See Figure~\ref{F:Rotation}.) 
If $\TT, \TTT$ are trees, we say that $\TTT$ is obtained from~$\TT$ by a
\emph{positive rotation}, or, equivalently, that $(\TT, \TTT)$ is a
\emph{positive base pair}, if there exists an inner node~$\vv$ of~$\TT$
such that this $\vv$-subtree of~$\TT$ has the form $\TT_1 \op (\TT_2
\op
\TT_3)$ and $\TTT$ is obtained from~$\TT$ by replacing the
$\vv$-subtree with $ (\TT_1 \op \TT_2) \op \TT_3$. In this case, we say
that $(\TTT, \TT)$ is a \emph{negative} base pair.
\end{defi}

\begin{figure}[htb]
\begin{picture}(75,33)(0,-1)
\put(0,-0.5){\includegraphics{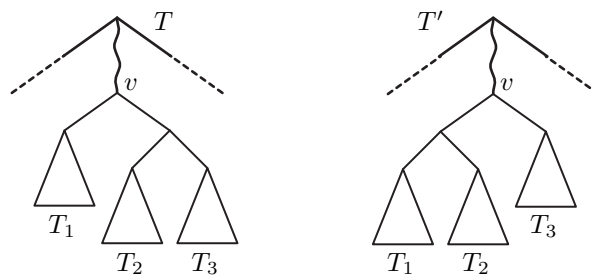}}
\put(20,31){$\TT$}
\put(16,23){$\vv$}
\put(6,4){$\TT_1$}
\put(15,-1){$\TT_2$}
\put(25,-1){$\TT_3$}

\put(55,31){$\TTT$}
\put(66,23){$\vv$}
\put(51,-1){$\TT_1$}
\put(61,-1){$\TT_2$}
\put(70,4){$\TT_3$}
\end{picture}
\caption{\sf\smaller A base pair: $\TTT$ is obtained from~$\TT$ by
rotating some subtree of~$\TT$ that can be expressed as $\TT_1
\op (\TT_2 \op \TT_3)$ to the corresponding $(\TT_1 \op \TT_2) \op
\TT_3$ (positive rotation)---or vice versa (negative rotation).}
\label{F:Rotation}
\end{figure}

By construction, rotations preserve the size of a tree. Conversely, it is easy
to see that, if $\TT$ and~$\TTT$ are trees with the same size, 
there exist a finite sequence of rotations that transforms~$\TT$
into~$\TTT$---see for instance Remark~\ref{R:Path} below---so a natural
notion of distance appears.

\begin{defi}
If $\TT, \TTT$ are equal size trees, the \emph{rotation distance}
between~$\TT$ and~$\TTT$, denoted~$\dist(\TT, \TTT)$, is the minimal
number of rotations needed to transform~$\TT$ into~$\TTT$. For $\nn
\ge 1$, we define $\dd(\nn)$ to be the maximum of $\dist(\TT, \TTT)$
for $\TT, \TTT$ of size~$\nn$.
\end{defi}

By definition, we have $\dist(\TT, \TTT) = 1$ if and only if $(\TT,
\TTT)$ is a base pair. As proved in~\cite[Lemma~2]{STT}, the
inequality $\dist(\TT, \TTT) \le 2\nn-6$ holds for all size~$\nn$
trees~$\TT, \TTT$ for~$\nn > 10$, and, therefore, we
have\footnote{contrary to~\cite{STT}, where notation changes from
Section~2.3, we stick to the convention that $\nn$ (and not $\nn-2$)
denotes the size of the reference trees} $\dd(\nn) \le 2\nn-6$ for $\nn
> 10$.

\subsection{Triangulations}

As explained in~\cite{STT}, there exists a one-to-one correspondence
between the triangulations of an $(\nn+2)$-gon and size~$\nn$ trees:
having chosen a distinguished edge, one encodes a triangulation by
the dual graph, a tree that becomes rooted once a distinguished
edge has been fixed (see Figure~\ref{F:Coding}).
Then performing one rotation corresponds to performing one flip in the
associated triangulations, this meaning that some pattern
\VR(9,7)\begin{picture}(15,0)(0,6)
\put(0,0){\includegraphics[scale=0.7]{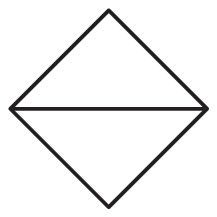}}
\end{picture}
is replaced with
\VR(9,7)\begin{picture}(15,0)(0,6)
\put(0,0){\includegraphics[scale=0.7]{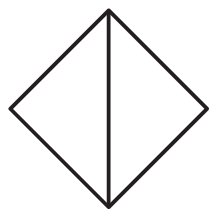}}
\end{picture},\ie, the diagonals are exchanged in the quadrilateral
made by two adjacent triangles. So the rotation
distance between two trees of size~$\nn$ is also the flip distance between
the corresponding triangulations of an $(\nn+2)$-gon, and the
number~$\dd(\nn)$ is the maximal flip distance between two
triangulations of an $(\nn+2)$-gon.

\begin{figure}[htb]
\begin{picture}(95,19)(0,0)
\put(0,0){\includegraphics[scale=1]{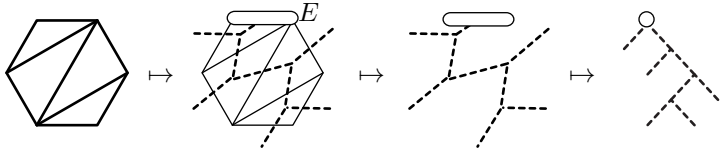}}
\put(40,17){$E$}
\put(20,9){$\mapsto$}
\put(48,9){$\mapsto$}
\put(76,9){$\mapsto$}
\end{picture}
\caption{\sf\smaller Coding a triangulation of an $(\nn+2)$-gon by a
size~$\nn$ tree: choose a distinguished edge~$E$, and hang the graph
dual to the triangulation under the vertex corresponding to~$E$.}
\label{F:Coding}
\end{figure}

\subsection{Associahedra}
\label{S:Associahedron}

For each~$\nn$, we have  a binary relation on size~$\nn$
trees, namely being at rotation distance~$1$. It is natural
to introduce the graph of this relation.

\begin{defi}
(See Figure~\ref{F:K4Double}.)
For $\nn \ge 1$, the \emph{associahedron}~$\Ass\nn$ is the (unoriented)
graph  whose vertices are size~$\nn$ trees and whose edges
are base pairs.
\end{defi}

The number of vertices of~$\Ass\nn$ is $\frac1{2\nn}{2\nn \choose
\nn}$, the $\nn$th Catalan number, and every vertex in~$\Ass\nn$ has
degree~$\nno$. The fact that any two trees of the same size are connected
by a sequence of rotations means that $\Ass\nn$ is a connected graph.
For all size~$\nn$ trees~$\TT, \TTT$, the number~$\dist(\TT, \TTT)$ is
the edge-distance between~$\TT$ and~$\TTT$ in~$\Ass\nn$, and the
number~$\dd(\nn)$ is the diameter of~$\Ass\nn$.

The name ``associahedron'' stems from the fact that, when we
decompose trees as iterated $\op$-products, performing a rotation
at~$\vv$ means applying the associativity law $\xx \op (\yy \op
\zz) = (\xx \op \yy) \op \zz$ to the $\vv$-subtree.

\begin{figure}[htb]
\begin{picture}(119,55)(0,0)
\put(0,0){\includegraphics[scale=0.7]{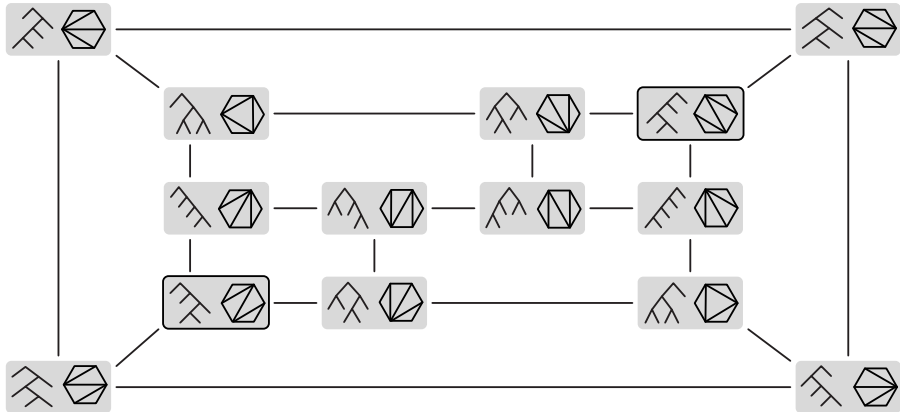}}
\end{picture}
\caption{\sf\smaller The associahedron~$K_4$: the vertices are the
fourteen trees of size~$4$, or, equivalently, the fourteen triangulations of
a hexagon, and the edges connect trees lying at rotation
distance~$1$ or, equivalently, triangulations lying at flip
distance~$1$. The diameter~$\dd(4)$
of~$\Ass4$ is~$4$. The thin trees $\Sp{1100}$ and~$\Sp{0011}$
(framed) are typical examples of size~$4$ trees at distance~$4$.}
\label{F:K4Double}
\end{figure}

Taking the sign of rotations into account, \ie, distinguishing whether
associativity is applied from $\xx \op (\yy \op \zz)$ to $(\xx \op \yy)
\op \zz$, or in the other direction, amounts to orienting the edges of
associahedra, as shown in Figure~\ref{F:K4-2} below. This orientation
defines a partial ordering on~$\Ass\nn$, which admits the right
comb~$\Sp{1^\nn}$ as a minimum and the left comb~$\Sp{0^\nn}$ as a
maximum. This partial ordering is known to be a lattice, the Tamari
lattice~\cite{Tam, HuT, Gey, Sun}. Let us mention that alternative lattice
orderings on~$\Ass\nn$ are constructed in~\cite{Rea} and~\cite{Kra}.

\section{Using positions}
\label{S:Positions}

Hereafter we address the problem of establishing lower bound for the
rotation distance $\dist(\TT, \TTT)$ between two trees of size~$\nn$, \ie,
to prove that any path from~$\TT$ to~$\TTT$ through~$\Ass\nn$
has length at least~$\ell$ for some~$\ell$. Any such result immediately
implies a lower bound for the diameter of the corresponding
associahedron. 

To this end, we have to analyze the rotations that lead
from~$\TT$ to~$\TTT$ and, for that, we need a way to specify a rotation
precisely. Exactly as in the case of permutations and their decompositions
into transpositions (instances of commutativity), we can specify a rotation
(\ie, an instance of associativity) by taking into account either the position
where the rotation occurs, or the names of the elements that are rotated. In
this section, we develop the first approach, based on positions.

\subsection{The address of a rotation}
\label{S:Address}

The position of a vertex~$\vv$ in a tree~$\TT$ can be unambiguously
specified using a finite  sequence of~$0$'s and~$1$'s that describes the
path from the root of~$\TT$ to~$\vv$, using~$0$ for forking to the left
and~$1$ for forking to the right. Such a sequence will be called a (binary)
\emph{address}. The set of all binary addresses will be denoted
by~$\Add$. The address of the root is the \emph{empty} sequence,
denoted~$\ea$. So, for instance, the addresses of the three inner nodes
of the tree of~\eqref{E:Tree} are $\ea, 1, 10$, whereas the addresses of
its four leaves are $0, 100, 101$, and~$11$:
$$\VR(6,4)\begin{picture}(14,0)(0,4)
\put(0,0){\includegraphics{Tree.eps}}
\put(4,10.5){$\scriptstyle\ea$}
\put(-1,6.5){$\scriptstyle0$}
\put(7.8,7){$\scriptstyle1$}
\put(0,3.5){$\scriptstyle10$}
\put(10.5,3.5){$\scriptstyle11$}
\put(-1.5,-1.5){$\scriptstyle100$}
\put(5,-1.5){$\scriptstyle101$}
\end{picture}. $$

We deduce a natural indexation of subtrees by  addresses. For
$\TT$ a tree and $\a$ a sufficiently short binary address, the
\emph{$\a$th subtree of~$\TT$},  denoted~$\sub\a\TT$, is  the
subtree of~$\TT$ whose root is the vertex that has
address~$\a$. Formally,
$\sub\a\TT$ is recursively defined by the rules
\begin{equation}
\label{E:Subtree}
\sub\ea\TT = \TT
\text{\ for every~$\TT$, \quad and \quad} 
\begin{cases}
\sub{0\a}\TT = \sub\a{{\TT_0}}\\
\sub{1\a}\TT = \sub\a{{\TT_1}}
\end{cases}
\text{\ for $\TT = \TT_0 \op \TT_1$}.
\end{equation}
Note that $\sub\a\TT$ exists and has positive size if and only if $\a$ is
the address of an inner node of~$\TT$, and it exists and has size~$0$
if and only if $\a$ is the address of a leaf of~$\TT$.

With such an indexation, we naturally attach an address with each
rotation.

\begin{defi}
We say that a positive base pair~$(\TT, \TTT)$ has
\emph{address}~$\a^+$ if $\a$ is the address in~$\TT$ (and in~$\TTT$)
of the root of the subtree that is rotated between~$\TT$ and~$\TTT$. The
address of the symmetric pair~$(\TTT,
\TT)$ is declared to be~$\a^-$.
\end{defi}

For instance, in the positive base pair
\begin{equation}
\label{E:Pair}
\VR(11,9)\begin{picture}(34,0)(0,10)
\put(0,0){\includegraphics{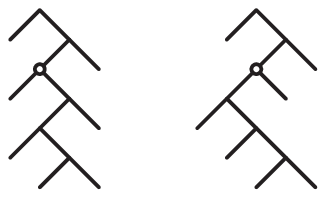}}
\put(0,13){$\scriptstyle10$}
\put(22,13){$\scriptstyle10$}
\put(6.5,19){$\TT$}
\put(21,19){$\TTT$}
\end{picture},
\end{equation}
the rotation involves the subtrees of~$\TT$ and~$\TTT$ whose roots have
address~$10$, hence the address of~$(\TT, \TTT)$ is declared to
be~$10^+$. See Figure~\ref{F:K4-2} for more examples.

\begin{figure}[htb]
\begin{picture}(88,63)(0,0)
\def\AAA#1{{#1}^+}
\put(-1,-1){\includegraphics[scale=0.8]{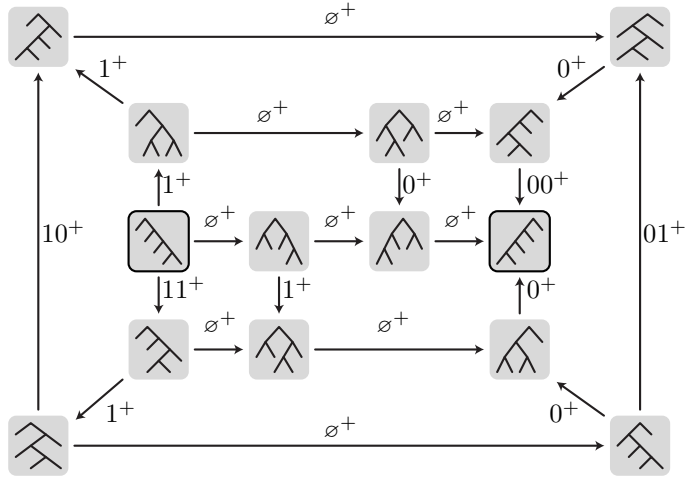}}
\put(42,5){$\AAA\ea$}
\put(42,33){$\AAA\ea$}
\put(42,60){$\AAA\ea$}
\put(26,33){$\AAA\ea$}
\put(58,33){$\AAA\ea$}
\put(57,47){$\AAA\ea$}
\put(33,47){$\AAA\ea$}
\put(26,19){$\AAA\ea$}
\put(49,19){$\AAA\ea$}
\put(20.5,37.5){$\AAA1$}
\put(52.5,37.5){$\AAA0$}
\put(69,37.5){$\AAA{00}$}
\put(20.5,23.5){$\AAA{11}$}
\put(36.5,23.5){$\AAA1$}
\put(69,23.5){$\AAA0$}
\put(4.5,31){$\AAA{10}$}
\put(84.5,31){$\AAA{01}$}
\put(13,7){$\AAA1$}
\put(72,7){$\AAA0$}
\put(12,53){$\AAA1$}
\put(73,53){$\AAA0$}
\end{picture}
\caption{\sf\smaller Orienting the edges of the associahedron~$\Ass4$
yields the Tamari lattice, in which the right comb $\Sp{1111}$ is minimal,
and the left comb $\Sp{0000}$ is maximal. Then taking positions into
account provides for each edge a label that is a binary address.}
\label{F:K4-2}
\end{figure}

The idea we shall develop in the sequel is to obtain lower bounds
$\dist(\TT, \TTT) \ge \ell$ by proving that at least $\ell$~addresses of
some prescribed type necessarily occur in any sequence of base pairs
connecting~$\TT$ to~$\TTT$, \ie, in any path from~$\TT$ to~$\TTT$
through the corresponding associahedron.

\subsection{Connection with Thompson's group~$F$}

To implement the above idea, it is convenient to view rotations as a partial
action of Thompson's group~$F$ on trees.

We recall from~\cite{CFP} that Thompson's group~$F$ consists of all
increasing piecewise linear self-homeomorphisms of~$[0,1]$ with dyadic
slopes and discontinuities of the derivative at dyadic points. There is a
simple correspondence between the elements of~$F$ and pairs of trees
of equal size. 

\begin{defi}
For $\pp \le \qq$ in~$\RRR$ and $\TT$ a size~$\nn$ tree, we
recursively define a partition~$\Part\TT\pp\qq$ of the real
interval~$[\pp,
\qq]$ into
$\nn+1$ adjacent intervals by
\begin{equation}
\Part\ut\pp\qq = \{[\pp, \qq]\}, 
\mbox{\quad and \quad}
\Part\TT\pp\qq = \Part{\TT_0}\pp{\frac{\pp+\qq}2}
\cup \Part{\TT_1}{\frac{\pp+\qq}2}\qq.
\end{equation}
Then, for $\TT, \TTT$ trees with equal sizes, we define~$\Th\TT{\TTT}$
to be the element of~$F$ that homothetically maps the $\ii$th interval
of~$\Part\TT01$ to the $\ii$th interval of~$\Part\TTT01$ for each~$\ii$.
\end{defi}

\begin{exam}
\label{X:Thompson}
Assume $\TT = \Sp{11}$ and $\TTT = \Sp{00}$. The partitions of~$[0,1]$
respectively associated with~$\TT$ and~$\TTT$ are
$$\Part\TT01 = \{[0,\frac12], [\frac12, \frac34], [\frac34, 1]\}
\mbox{, and }
\Part\TTT01 = \{[0,\frac14], [\frac14, \frac12], [\frac12, 1]\}.$$
Therefore, $\Th\TT\TTT$ is the element of Thompson's group~$F$
whose graph is
\begin{equation}
\VR(15,13)\begin{picture}(25,12)(0,12)
\put(0,0){\includegraphics{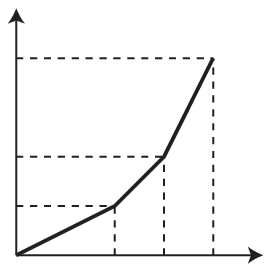}}
\end{picture}
\end{equation}
\ie, the element denoted~$A$ in~\cite{CFP} (or $\xx_0$ in most recent
references).
\end{exam}

Infinitely many different pairs of trees represent a given element of~$F$:
if $\TT$ and $\TTT$ are obtained from~$\TT_0$ and~$\TTT_0$
respectively by replacing the $\kk$th leaf with a caret, then the
subdivisions~$\Part\TT01$ and~$\Part\TTT01$ are obtained
from~$\Part{\TT_0}01$ and~$\Part{\TTT_0}01$ by subdividing the
$\kk$th intervals, and we have $\Th\TT\TTT = \Th{\TT_0}{\TTT_0}$.
However, for each element~$\gg$ of~$F$, there exists a unique
\emph{reduced} pair of trees~$(\dom\gg, \tar\gg)$ satisfying
$\Th{\dom\gg}{\tar\gg} = \gg$, where $(\TT, \TTT)$ is called reduced if
it is obtained from no other pair $(\TT_0, \TTT_0)$ as above
\cite[Section~2]{CFP}.

We can then introduce a partial action of the group~$F$ on trees as
follows.

\begin{defi}
\label{D:Action}
For $\TT$ a tree and $\gg$ an element of the group~$F$, we define $\TT
\act \gg$ to be the unique tree~$\TTT$ satisfying $\Th\TT\TTT = \gg$, if
if exists.
\end{defi}

This is a partial action: $\TT \act \gg$ need not be defined for
all~$\TT$ and~$\gg$. By construction, $\TT \act \gg$ exists if and only if
the partition~$\Part\TT01$ refines the partition~$\Part{\dom\gg}01$, \ie,
equivalently, if and only if the tree~$\dom\gg$ is included in~$\TT$.
However, for each pair of trees of equal size $(\TT, \TTT)$, there exists an
element~$\gg$ satisfying $\TT \act \gg = \TTT$, namely~$\Th\TT\TTT$,
and the rules for an action on the right are obeyed: if $\TT \act
\gg$ and $(\TT \act \gg) \act \hh$ are defined, then $\TT \act (\gg\hh)$
is defined and it is equal to $(\TT \act
\gg) \act \hh$---we assume that the product in~$F$ corresponds to
reverse composition: $\gg\hh$ means ``$\gg$, then~$\hh$''. For our
purpose, the main point is the following direct consequence of
Definition~\ref{D:Action}.

\begin{lemm}
\label{L:Free}
The partial action of~$F$ on trees is free: if $\TT \act \gg$ and
$\TT \act \gg'$ are defined and equal, then $\gg = \gg'$ holds.
\end{lemm}

\subsection{The generators~$\AAA\a$ of the group~$F$}

It is now easy to describe the rotations with address~$\a^\pm$ in terms of
the action of an element of~$F$. 

\begin{defi}
For each address~$\a$ in~$\Add$, we put $\AAA\a
= \Th{\Sp{\a11}}{\Sp{\a00}}$. The family of all elements~$\AAA\a$ is
denoted~$\AAAA$.
\end{defi}

\begin{exam}
The element~$\AAA\ea$ is $\Th{\Sp{11}}{\Sp{00}}$,
hence it is the element~$A$ (or~$x_0$) considered in
Example~\ref{X:Thompson}. More generally, if $\a$ is $1^\ii$, \ie, $\a$
consists of~$1$ repeated
$\ii$~times, $\AAA\a$ is the element usually denoted~$\xx_\ii$, which
corresponds to applying associativity at the $\ii$th position on the right
branch of the tree. Viewed as a function of~$[0,1]$ to itself, $\AAA\a$ is
the identity on~$[0, 1-2^{-\ii}]$ and its graph on $[1- 2^{-\ii}, 1]$ is that
of~$\xx_0$ contracted by a factor~$2^{\ii}$.
\end{exam}

By definition, $\AAA\a$ corresponds to applying associativity at 
position~$\a$, and, therefore, we have the
following equivalence, whose verification is straightforward.

\begin{lemm}
For all trees $\TT, \TTT$ with the same size, the following
are equivalent:

$(i)$ $(\TT, \TTT)$ is a base pair with address~$\a^\pm$;

$(ii)$ $\TTT = \TT \act \AAA\a^{\pm1}$ holds.
\end{lemm}

It is well-known that the group~$F$ is generated by the
elements~$\AAA\ea$ and~$\AAA1$, hence, a fortiori, by the whole
family~$\AAAA$. For~$\gg$ in~$F$, we denote by~$\ell_{\!\AAAA}(\gg)$
the length of~$\gg$ with respect to~$\AAAA$, \ie, the length of the
shortest expression of~$\gg$ as a product of letters~$\AAA\a$
and~$\AAA\a\inv$. Then the connection between the rotation distance
and the length function~$\ell_\AAAA$ is very simple.

\begin{prop}
For all trees $\TT, \TTT$ with the same size, we
have
\begin{equation}
\label{E:Distance}
\dist(\TT, \TTT) = \ell_{\!\AAAA}(\Th\TT{\TTT}).
\end{equation}
\end{prop}

\begin{proof}
Assume that $(\TT_0, ..., \TT_\ell)$ is path from~$\TT$
to~$\TTT$ in~$\Ass{\size\TT}$, \ie, we have $\TT_0 =\nobreak \TT$,
$\TT_\ell =
\TTT$, and
$(\TT_\rr, \TT_\rrp)$ is a base pair for each~$\rr$. Let
$\a_\rr^{\ee_\rr}$ be the address of ~$(\TT_\rr, \TT_\rrp)$. By
construction, we have
$$\TT \act (\AAA{\a_1}^{\ee_1} ... \AAA{\a_\ell}^{\ee_\ell}) = \TTT,$$
hence $\Th\TT\TTT = \AAA{\a_1}^{\ee_1} ... \AAA{\a_\ell}^{\ee_\ell}$
by Lemma~\ref{L:Free}, and, therefore, $\ell_{\!\AAAA}(\Th\TT{\TTT})
\le \ell$.

Conversely, assume that $\AAA{\a_1}^{\ee_1} ...
\AAA{\a_\ell}^{\ee_\ell}$ is an expression of the element~$\Th\TT\TTT$
of~$F$ in terms of the generators~$\AAA\a$. It need not be true that
$\TT \act (\AAA{\a_1}^{\ee_1} ... \AAA{\a_\ell}^{\ee_\ell})$ is defined,
but we can always find an extension~$\widehat\TT$ of~$\TT$ (a tree
obtained from~$\TT$ by adding more carets) such that 
$\widehat\TT \act (\AAA{\a_1}^{\ee_1} ... \AAA{\a_\ell}^{\ee_\ell})$ is
defined and equal to some extension~$\widehat\TTT$ of~$\TTT$. Hence
we have $\dist(\widehat\TT, \widehat\TTT) \le \ell$. Now, anticipating
on Section~\ref{S:Collapsing}, there exists a set~$\II$ such that
$\TT$ and $\TTT$ are obtained from~$\widehat\TT$ and
$\widehat\TTT$ respectively by collapsing the labels of~$\II$. As will
follow from Lemma~\ref{L:CollPair}, this implies $\dist(\TT, \TTT) \le
\dist(\widehat\TT, \widehat\TTT) \le \ell$.
\end{proof}

\subsection{Presentation of~$F$ in terms of the generators~$\AAA\a$}

We are thus left with the question of determining the length of an element
of the group~$F$ in terms of the generators~$\AAA\a$. Formally, this
problem is closed to similar length problems for which solutions
are known. In~\cite{For} and~\cite{BeB}, explicit combinatorial methods
for computing the length of an element of~$F$ with respect to the
generating family~$\{\xx_0, \xx_1\}$, \ie, $\{\AAA\ea,
\AAA1\}$ with the current notation, are given. Similarly, the unique
normal form of~\cite[Theorem~2.5]{CFP} is geodesic with respect to the
generating family~$\{\xx_\ii \mid \ii \ge 0\}$ and, therefore, there exists
an explicit combinatorial method for computing the length with respect to
that generating family, \ie, with respect to $\{\AAA\ea, \AAA1, \AAA{11},
...\}$. Thus, one might expect a similar method
for computing the length with respect to the family $\AAAA$ of
all~$\AAA\a$'s. Unfortunately, no such method is known at the
moment, and we can only obtain coarse inequalities.

The first step is to determine a presentation of the group~$F$ in terms of
the (redundant) family~$\AAAA$. As we know a
presentation of~$F$ from the~$\xx_\ii$'s, it is sufficient to add the
definitions of the elements~$\AAA\a$ for~$\a$ not a power of~$1$.
Actually, a much more symmetric presentation exists.

If $\a, \b$ are binary addresses, we say that $\a$ is a \emph{prefix} of~$\b$, denoted $\a \prefe \b$, if $\b = \a\g$
holds for some address~$\g$.

\begin{lemm}
[{\cite[Prop. 4]{Dfg} or \cite[Prop. 2.13]{Dhb}}]
\label{L:Pres}
In terms of~$\AAAA$, the group~$F$ is presented by the
following relations, where $\a, \b$ range over~$\Add$,
\begin{align}
\label{E:Pentagon}
\AAA\a^2 
&= \AAA{\a1} \cdot \AAA\a \cdot \AAA{\a0},\qquad\qquad\\
\label{E:Rel2}
\AAA{\a0\b} \cdot \AAA\a
&= \AAA\a \cdot \AAA{\a00\b},\\
\AAA{\a10\b} \cdot \AAA\a
&= \AAA\a \cdot \AAA{\a01\b},\\
\AAA{\a11\b} \cdot \AAA\a
&= \AAA\a \cdot \AAA{\a1\b},\\
\label{E:Rel5}
\AAA\b \cdot \AAA\a
&= \AAA\a \cdot \AAA\b
\rlap{\mbox{\quad if neither $\a \prefe \b$ nor $\b \prefe \a$
holds.}}
\end{align}
\end{lemm}

Relations~\eqref{E:Pentagon} are MacLane--Stasheff pentagon relations,
whereas \eqref{E:Rel2}--\eqref{E:Rel5} are quasi-commutation relations
with an easy geometric meaning.

For a given pair of trees~$(\TT, \TTT)$, it is not difficult to find an
expression of~$\Th\TT{\TTT}$ by an $\AAA\a$-word, \ie, a word in
the letters~$\AAA\a^{\pm1}$---see
Remark~\ref{R:Path} below. In general, nothing guarantees that this
expression~$\ww$ is geodesic, and we can only deduce an upper bound
$\dist(\TT, \TTT) \le \lg\ww$. However, we can also obtain lower bounds
by using invariants of the relations of Lemma~\ref{L:Pres}.

\begin{lemm}
\label{L:Invariant}
For~$\ww$ an $\AAA\a$-word, let $\II(\ww)$ be the algebraic sum of
the exponents of the letters~$\AAA\a$ with $\a$ containing no~$0$. Then
$\II$ is invariant under the relations of Lemma~\ref{L:Pres}.
\end{lemm}

\begin{proof}
A simple inspection. For instance, if $\a$ contains no~$0$, there are two
letters of the form~$\AAA{1^\ii}$ on both sides of~\eqref{E:Pentagon},
whereas, if $\a$ contains at least one~$0$, there are no such letter on
either side. Also notice that the invariance property holds for the implicit
free group relations
$\AAA\a \cdot \AAA\a\inv = 1$ and $\AAA\a\inv \cdot \AAA\a = 1$.
\end{proof}

\begin{prop}
\label{P:Geodesic}
An $\AAA\a$-word that only contains letters~$\AAA{1^\ii}$ with
positive exponents is geodesic.
\end{prop}

\begin{proof}
Using $\lg\ww$ for the length (number of letters) of a word~$\ww$, we
have $\II(\ww) \le \lg\ww$ for every $\AAA\a$-word~$\ww$, by
definition of~$\II$. Now, if $\ww$ contains only letters~$\AAA{1^\ii}$ with
positive exponents, then we have $\II(\ww) = \lg\ww$. By
Lemma~\ref{L:Invariant}, we deduce 
\linebreak
$\lg\ww = \II(\ww) = \II(\ww') \le
\lg{\ww'}$ for each
$\AAA\a$-word that is equivalent to~$\ww$ under the relations of
Lemma~\ref{L:Pres}.
\end{proof}

\begin{rema}
By contrast, it is \emph{not} true that a positive $\AAA\a$-word, \ie, an
$\AAA\a$-word in which all letters have positive exponents (no
$\AAA\a\inv$) need to be geodesic, or even quasi-geodesic: for $\pp \ge
1$, let us write $\AAA\a^{(\pp)}$ for $\AAA{\a1^{\pp-1}}
\AAA{\a1^{\pp-2}} ... \AAA{\a1} \AAA\a$. Then one can check that, for
each~$\pp$, the positive $\AAA\a$-word
$$\AAA\ea^{(\pp)} 
\AAA{01}^{(\pp-1)} 
\AAA{0101}^{(\pp-2)} 
\ ... \ 
\AAA{(01)^{\pp-2}}^{(2)}
\AAA{(01)^{\pp-1}}^{(1)},$$
which has length $\pp(\pp+1)/2$, is equivalent to the $\AAA\a$-word
$$\AAA\ea^{(\pp-1)} \AAA\ea \AAA0\inv \AAA{01} \AAA{010}\inv \ ... \
\AAA{(01)^{\pp-2}} \AAA{(01)^{\pp-2}0}\inv
\AAA{(01)^{\pp-1}},$$
which has length~$3\pp-2$. In other words, the
submonoid of~$F$ generated by the elements~$\AAA\a$ is not
quasi-isometrically embedded in the group~$F$.
\end{rema}

As an application, we can determine the distance from a right comb
to any tree.

\begin{prop}
\label{P:Comb}
For each tree~$\TT$ of size~$\nn$, we have
\begin{equation}
\label{E:Comb}
\dist(\Sp{1^\nn}, \TT) = \nn - \hR(\TT),
\end{equation}
where $\hR$ denotes the length of the rightmost branch.
\end{prop}

\begin{proof}
Using $\sh_1$ for the word homomorphism that maps~$\AAA\a$
to~$\AAA{1\a}$ for each~$\a$, we recursively define an
$\AAA\a$-word~$\cha\TT$ by
\begin{equation}
\label{E:Char}
\cha\TT = 
\begin{cases}
\e \mbox{ (the empty word)}
&\mbox{for $\TT = \ut$},\\
\cha{\TT_0} \cdot \AAA{1^{\hR(\TT_0)-1}} ... \AAA1 \AAA\ea \cdot
\sh_1(\cha{\TT_1})
&\mbox{for $\TT = \TT_0 \op \TT_1$}.
\end{cases}
\end{equation}
An easy induction shows that, for each size~$\nn$ tree~$\TT$, the length
of~$\cha\TT$ is $\nn - \hR(\TT)$, and that we have
$\Sp{1^\nn} \act \cha\TT = \TT$, \ie, $\cha\TT$ provides a
distinguished way to go from the right comb~$\Sp{1^\nn}$ to~$\TT$.

Now, we see on~\eqref{E:Char} that $\cha\TT$ exclusively consists
of letters~$\AAA{1^\ii}$ with a positive exponent. Hence, by
Proposition~\ref{P:Geodesic}, $\cha\TT$ is geodesic, and we
deduce 
$$\dist(\Sp{1^\nn}, \TT)  = \lg{\cha\TT} =  \nn - \hR(\TT).
\eqno{\square}$$
\def\qed{\relax}
\end{proof}

\begin{coro}
\label{C:Comb}
For each~$\nn$, we have $\dd(\nn) \ge \nn-1$.
\end{coro}

\begin{proof}
Applying~\eqref{E:Comb} when $\TT$ is the left comb~$\Sp{0^\nn}$---or
any size~$\nn$ term with right height~$1$---gives $\dist(\Sp{1^\nn},
\TT) = \nn-1$.
\end{proof}

As can be expected, other proofs of the previous modest result can be
given, for instance by counting left-oriented edges in the trees: our
purpose in stating Proposition~\ref{P:Comb} is just to illustrate the
general principle of using Thompson's group~$F$ and its presentation
from the $\AAA\a$'s.

\begin{rema}
\label{R:Path}
The proof of Proposition~\ref{P:Comb} implies that, for each
pair~$(\TT, \TTT)$ of size~$\nn$ trees, the $\AAA\a$-word $\cha\TT\inv
\, \cha{\TTT}$ is an explicit expression of~$\Th\TT{\TTT}$ of length at
most $2\nn-2$. It corresponds to a distinguished path from~$\TT$
to~$\TTT$ via the right comb~$\Sp{1^\nn}$ in the
associahedron~$\KK_\nn$.
\end{rema}

\subsection{Addresses of leaves}
\label{S:AddLeaves}

We turn to another way of using addresses to prove lower bounds on the
rotation distance, namely analyzing the way the addresses of the leaves are
modified in rotations.

\begin{defi}
For~$\TT$ a tree and $1 \le \ii \le \size\TT+1$, we denote
by~$\add\TT\ii$ the address of the $\ii$th leaf of~$\TT$ in the
left-to-right enumeration of leaves.
\end{defi}

Equivalently, we can attribute labels~$1$ to~$\nnp$ to the
leaves of each size~$\nn$ tree~$\TT$ enumerated from left to right and,
then,
$\add\TT\ii$ is the address where the label~$\ii$ occurs in~$\TT$. For
instance, if
$\TT$ is the (thin) tree of~\eqref{E:Tree}, then the labelling of the leaves
of~$\TT$ is
$$\VR(6,5)\begin{picture}(13,0)(0,4)
\put(0,0){\includegraphics{Tree.eps}}
\put(-1,5){$\scriptstyle1$}
\put(10.5,2){$\scriptstyle4$}
\put(-0.5,-1.5){$\scriptstyle2$}
\put(7,-1.5){$\scriptstyle3$}
\end{picture},$$
and we find $\add\TT1 = 0$,  $\add\TT2 = 100$, $\add\TT3 = 101$, and
$\add\TT4 = 11$.

The idea now is that, for each~$\ii$, we can follow the
parameter~$\add\TT\ii$ when rotations are applied.

\begin{lemm}
\label{L:Addresses}
Assume that $(\TT, \TTT)$ is a base pair and $1 \le
\ii \le \size\TT+1$ holds. Then $\add\TTT\ii$ is equal
to~$\add\TT\ii$ or it is obtained from~$\add\TT\ii$ by one of the
following transformations, hereafter called \emph{special}:
adding or removing one~$0$, adding or removing one~$1$, replacing some
subword~$10$ with~$01$ or vice versa.
\end{lemm}

\begin{proof}
Assume that the address of~$(\TT, \TTT)$
is~$\a^+$ and $\add\TT\ii = \g$ holds. Then four cases may occur. 
If $\a$ is not a prefix of~$\g$, then we have $\add\TTT\ii = \g$.
Otherwise, exactly one of $\a0 \prefe \g$, $\a10 \prefe \g$, or $\a11
\prefe \g$ holds. If $\a0 \prefe \g$ holds, say $\g = \a0\b$, then we have
$\add\TTT\ii = \a00\b$, obtained from~$\g$ by adding one~$0$. If
$\a10 \prefe \g$ holds, say $\g = \a0\b$, then we have
$\add\TTT\ii = \a01\b$, obtained from~$\g$ by replacing one~$10$
with~$01$. Finally, if $\a11 \prefe \g$ holds, say $\g = \a11\b$, then we
have $\add\TTT\ii = \a1\b$, obtained from~$\g$ by removing one~$1$. 
The results are symmetric for a negative base pair.
\end{proof}

So Lemma~\ref{L:Addresses} shows that the parameters~$\add\TT\ii$
cannot change too fast, which can be easily translated into lower bounds
on the rotation distance. For $\a, \g$ two binary addresses, we denote
by~$\Card\a\g$ the number of occurrences of~$\a$ in~$\g$.

\begin{lemm}
\label{L:Address}
For~$\g, \g'$ in~$\Add$, define
$$\d(\g, \g') = 
\Abs{\Card0{\g'} - \Card0\g} 
+ \Abs{\Card1{\g'} - \Card1\g}
+ \Abs{\Card{10}{\g'} - \Card{10}\g}.$$
Then, for all trees~$\TT, \TTT$ of equal size, we have
\begin{equation}
\label{E:AddDist}
\dist(\TT, \TTT) \ge \max_{1 \le \ii \le \size\TT}
\d(\add\TT\ii, \add\TTT\ii).
\end{equation}
\end{lemm}

\begin{proof}
Lemma~\ref{L:Addresses} shows that, for each~$\ii$, one rotation
changes by at most one the value of $\d(\add\TT\ii, \add\TTT\ii)$.
Indeed, adding or removing one~$0$ can change the value of
$\Abs{\Card0{\add\TTT\ii} - \Card0{\add\TT\ii}}$ by one, but it
changes neither $\Abs{\Card1{\add\TTT\ii} - \Card1{\add\TT\ii}}$
nor $\Abs{\Card{10}{\add\TTT\ii} - \Card{10}{\add\TT\ii}}$.
Similarly, exchanging~$10$ and~$01$ once can change the value of
$\Abs{\Card{10}{\add\TTT\ii} - \Card{10}{\add\TT\ii}}$ by one, but
it changes neither $\Abs{\Card0{\add\TTT\ii} - \Card0{\add\TT\ii}}$
nor $\Abs{\Card1{\add\TTT\ii} - \Card1{\add\TT\ii}}$. Thus, if
$\d(\add\TT\ii, \add\TTT\ii)$ is~$\ell$, then at least~$\ell$ rotations
are needed to transform~$\TT$ into~$\TTT$.
\end{proof}

\begin{prop}
\label{P:Lower32}
For $\TT = (\TT_1 \op \ut) \op \TT_2$, with $\size{\TT_1} =
\size{\TT_2} = \ppm$, and $\TTT =\nobreak \Sp{(10)^\pp}$, we have
\begin{equation}
\label{E:Lower32}
\dist(\TT, \TTT) \ge 3\pp-2.
\end{equation}
\end{prop}

\begin{proof}
Both $\TT$ and $\TTT$ have size~$2\pp$. By construction, we have
$$\add\TT\ppp = 01
\mbox{\quad and \quad}
\add\TTT\ppp = (10)^\pp.$$
We have $\Card0{01} =  \Card1{01} = 1$, 
$\Card{10}{01} = 0$, and  
$\Card0{(10)^\pp} =  \Card1{(10)^\pp} =  \Card{10}{(10)^\pp} = \pp$,
whence
$\d(01, (10)^\pp) =3\pp-2$, and
\eqref{E:Lower32} follows from~\eqref{E:AddDist}.
\end{proof}

\begin{coro}
For each~$\nn$, we have $\dd(\nn) \ge \frac32\nn- \frac52$.
\end{coro}

\begin{proof}
Proposition~\ref{P:Lower32} gives $\dd(\nn) \ge
\frac32\nn -2$ for $\nn$ even. A similar argument gives $\dd(\nn) \ge
\frac32\nn- \frac52$ for $\nn$~odd.
\end{proof}

\subsection{Addresses of leaves (continued)}

The method can be refined to obtain a more precise evaluation of the
number of rotations needed to transform an address into a special given
address, typically one of the form~$1^\pp 0^\qq$. For $\g$ a binary
address, we denote by
$\pi(\g)$ the unique path in~$\Nat^2$ that starts on~$\Nat\times\{0\}$,
finishes on~$\{0\}\times\Nat$ and contains one edge~$(0,-1)$ for
each~$1$ in~$\g$ and one edge~$(1,0)$ for each~$0$ in~$\g$, following
the order of letters in~$\g$; see Figure~\ref{F:AddressBis}. Then we have
the following result.

\begin{lemm}
\label{L:AddressBis}
(See Figure~\ref{F:AddressBis}.) For $\g$ satisfying $\Card1\g
\le\nobreak \pp$ and $\Card0\g \le \qq$, put 
\begin{equation}
\label{E:AddressBis}
\ff(\pp, \qq, \g) = (\pp - \Card1\g) + (\qq - \Card0\g) + N(\g) + D(\g),
\end{equation}
where $N(\g)$ denotes the number of squares lying
below~$\pi(\g)$ and adjacent to a coordinate axis, and
$D(\g)$ is the distance in the $\Nat^2$-grid from $(1, 1)$ to the
region above~$\pi(\g)$. Then, for all trees~$\TT, \TTT$ satisfying
$\add\TT\ii = 1^\pp0^\qq$, $\Card1{\add\TTT\ii} \le \pp$, and
$\Card0{\add\TTT\ii} \le \qq$, we have
\begin{equation}
\dist(\TT, \TTT) \ge \ff(\pp, \qq, \add\TTT\ii)
\end{equation}
\end{lemm}

\begin{figure}[htb]
\begin{picture}(41,26)(0,0)
\put(-1,-1){\includegraphics[scale=1]{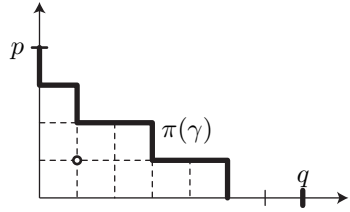}}
\put(-3,20){$\pp$}
\put(35,3){$\qq$}
\put(17,9){$\pi(\g)$}
\end{picture}
\caption{\sf\smaller Application of Lemma~\ref{L:AddressBis} to $\g =
101001001$ with $\pp = 4$ and $\qq = 7$: we find
$\pp - \Card1\g = 4-4 = 0$, $\qq - \Card0\g = 7 - 5 = 2$, the number of
squares below~$\pi(\g)$ touching one of the axes is~$7$, and the
distance from~$(1,1)$ to the region above~$\pi(\g)$ is~$1$, leading to
$\ff(4, 7, \g) = 0 + 2 + 7 + 1 = 10$. Hence, if
$\TT$ and $\TTT$ are trees in which, for some~$\ii$, the
$\ii$th leaf has address~$1^4 0^7$ in~$\TT$ and~$\g$
in~$\TTT$, we have $\dist(\TT, \TTT) \ge\nobreak 10$.}
\label{F:AddressBis}
\end{figure}

\begin{proof}[Proof (sketch)]
As in Lemma~\ref{L:Address}, one checks that $\ff(\pp, \qq, \g)$
decreases by at most one when a special transformation is applied
to~$\g$. By Lemma~\ref{L:Addresses}, this implies that $\ff(\pp, \qq,
\add\TTT\ii)$ decreases by at most one when a rotation is applied
to~$\TTT$. 
\end{proof}

\begin{prop}
\label{P:Bicomb}
For $\TT = \Sp{1^\pp0^\qq}$, $\TTT = \Sp{0^\qq1^\pp}$ with~$\pp,
\qq
\ge 1$, we have
\begin{equation}
\label{E:Bicomb}
\dist(\TT, \TTT) \ge \pp + \qq + \min(\pp, \qq) - 2.
\end{equation}
\end{prop}

\begin{proof}
As Figure~\ref{F:Bicomb} will show, we have $\add\TT\ppp =
1^\pp0^\qq$ and
$\add\TTT\ppp =\nobreak 0^\qq1^\pp$, so applying
Lemma~\ref{L:AddressBis} gives the lower  bound $\dist(\TT, \TTT) \ge
\pp + \qq + \min(\pp, \qq) - 2$. 
\end{proof}

(It is then easy to check that \eqref{E:Bicomb} is an equality by finding an
explicit path from~$\TT$ to~$\TTT$ involving the expected number of
rotations.) So we reobtain for~$\dd(\nn)$ a lower
bound~$\frac32\nn -2$ for even~$\nn$, and $\frac32\nn - \frac52$ for
odd~$\nn$. More refinements are possible, but approaches that only take 
one address at a time into account are unlikely to go beyond~$\frac32\nn
+ O(1)$.

\section{Using names}
\label{S:Names}

As the previous approach based on positions leads to limited
results only,  we now develop an alternative approach based on
\emph{names}. This is exactly similar to investigating a permutation not
in terms of the positions of the elements that are permuted, but in terms of
the names of the elements that have been permuted. Here we shall
associate with every base pair a name~$(\aa, \bb, \cc, \dd)^\pm$
consisting of four numbers and a sign. The principle for using names will
be the same as with positions: we prove $\dist(\TT, \TTT) \ge \ell$ by
showing that any path from~$\TT$ to~$\TTT$ through the
associahedron~$\Ass{\size\TT}$ must visit at lest $\ell$~edges whose
names satisfy some specific constraints. Here the constraints will involve
the so-called covering relation, a binary relation that connects
the leaves of the tree, taking the form ``the
$\ii$th leaf is covered by the $\jj$th leaf in~$\TT$'', denoted $\ii \cov\TT
\jj$. The basic observation is that, if the $\ii$th leaf is not covered by
the $\jj$th leaf in~$\TT$, but is covered in~$\TTT$, then any sequence of
rotations from~$\TT$ to~$\TTT$ must contain a base pair whose name
has a certain form, namely $(\aa, \bb, \cc, \dd)^+$ with $\cc = \jj$ and
$\aa \le \ii < \bb$. The rest of the paper consists in exploiting this
principle in more and more sophisticated ways.

\subsection{The name of a base pair}
\label{S:Name}

As in Section~\ref{S:AddLeaves}, we attach labels to the leaves of
a tree. For each label~$\ii$ occurring in a tree~$\TT$, we
denote by~$\add\TT\ii$ the address where~$\ii$ occurs
in~$\TT$. The only difference is that, in view of 
Section~\ref{S:Collapsing}, we shall not necessarily assume that the labels
used for a size~$\nn$ tree are $1$ to~$\nn+1$. However we
always assume that the labels increase  from left to right. So, for instance, 
$$\VR(6,4)\begin{picture}(13,0)(0,4)
\put(0,0){\includegraphics{Tree.eps}}
\put(-1,5){$\scriptstyle2$}
\put(-0.5,-1.5){$\scriptstyle5$}
\put(7,-1.5){$\scriptstyle6$}
\put(10.5,2){$\scriptstyle9$}
\end{picture}$$
is considered to be a legal labeling and, if $\TT$ is the above tree,
we would write $\add\TT2 = 0$ and $\add\TT6 = 101$. Formally, this
amounts to hereafter considering \emph{labeled} trees. However, every
unlabeled tree is identified with the labeled tree  where labels
are $1, \Ldots, \size\TT+1$.

Switching to labeled trees does not change anything to rotations and the
derived notions. If $\TT$ and~$\TTT$ are labeled trees, we say that
$(\TT, \TTT)$ is a base pair if $(\TTb, \TTTb)$ is a base pair, where
$\TTb$ and~$\TTTb$ the unlabeled trees underlying~$\TT$
and~$\TTT$, and, in addition, the same labels occur in~$\TT$
and~$\TTT$ (necessarily in the same order by our convention that labels
increase from left to right). As the associativity law does not change the
order of variables, this definition preserves the connection between
rotation and associativity.

We now attach to each base pair (of labeled trees) a name
that specifies the rotation in terms of the labels of leaves.

\begin{defi}
\label{D:Name}
(See Figures~\ref{F:Name} and~\ref{F:K4Names} below.) Assume that
$(\TT,
\TTT)$ is a positive base  pair. Let~$\a^+$ be the address of~$(\TT,
\TTT)$. Then, the
\emph{name} of $(\TT, \TTT)$, denoted  $\name\TT\TTT$, is defined to
be $(\aa, \bb, \cc, \dd)^+$, where

\quad $\aa$ is the unique label in~$\TT$ that satisfies $\add\TT\aa =
\a0^\pp$ for some~$\pp$,

\quad $\bb$ is the unique label in~$\TT$ that satisfies $\add\TT\bb =
\a10^\pp$ for some~$\pp$,

\quad $\cc$ is the unique label in~$\TT$ that satisfies $\add\TT\cc =
\a101^\pp$ for some~$\pp$,

\quad $\dd$ is the unique label in~$\TT$ that satisfies $\add\TT\dd =
\a1^\pp$ for some~$\pp$,

\noindent
In this case, the name of~$(\TTT, \TT)$ is defined to be $(\aa, \bb, \cc,
\dd)^-$. For $\kk = 1, ..., 4$, the
$\kk$th entry in~$\name\TT\TTT$ is denoted~$\Name\kk\TT\TTT$.
\end{defi} 

\begin{figure}[htb]
\begin{picture}(70,33)
\put(0,1.5){\includegraphics{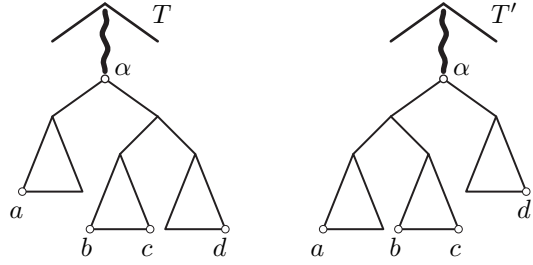}}
\put(19,30){$\TT$}
\put(64,30){$\TTT$}
\put(14,23){$\a$}
\put(59,23){$\a$}
\put(0,4){$\aa$}
\put(9.5,-1){$\bb$}
\put(17.5,-1){$\cc$}
\put(27,-1){$\dd$}
\put(40,-1){$\aa$}
\put(50.5,-1){$\bb$}
\put(58.5,-1){$\cc$}
\put(67.5,4){$\dd$}
\end{picture}
\caption{\sf\smaller Name of a base pair of trees:  $\Name1\TT\TTT$ and
$\Name4\TT\TTT$, \ie, $\aa$ and $\dd$, are the names of the extremal
leaves in the subtree involved in the rotation, whereas
$\Name2\TT\TTT$ and $\Name3\TT\TTT$, \ie, $\bb$ and $\cc$, are the
names of the  extremal leaves in the (nested) subtree that is actually moved
in the rotation.}
\label{F:Name}
\end{figure}

\begin{exam}
\label{X:Name}
Let us consider the pair of~\eqref{E:Pair} again, namely
\begin{equation}
\label{E:PairBis}
\VR(11,11)\begin{picture}(34,0)(0,10)
\put(0,0){\includegraphics{Pair.eps}}
\put(6.5,19){$\TT$}
\put(21,19){$\TTT$}
\put(-1,14){$\scriptstyle1$}
\put(-1,8){$\scriptstyle2$}
\put(-1,2){$\scriptstyle3$}
\put(2,-1){$\scriptstyle4$}
\put(11,-1){$\scriptstyle5$}
\put(11,5){$\scriptstyle6$}
\put(11,11){$\scriptstyle7$}
\put(21,14){$\scriptstyle1$}
\put(18,5){$\scriptstyle2$}
\put(21,2){$\scriptstyle3$}
\put(24,-1){$\scriptstyle4$}
\put(32,-1){$\scriptstyle5$}
\put(29,8){$\scriptstyle6$}
\put(32,11){$\scriptstyle7$}
\end{picture},
\end{equation}
With the default labels, the name of~$(\TT,
\TTT)$ is $(2, 3, 5, 6)^+$: indeed, $2$ and~$6$ are the extreme labels of
the subtree~$\sub{10}\TT$ involved in the rotation, whereas $3$
and~$5$ are the extreme labels in the subtree~$\sub{1010}\TT$ that is
actually moved. 
\end{exam}

By definition, when $\TT$ is given, the address of a base
pair~$(\TT, \TTT)$ determines its name. The converse is true as well: if
$(\TT, \TTT)$ has name~$(\aa, \bb, \cc, \dd)^\pm$, then the address
of~$(\TT, \TTT)$ is~$\a^\pm$, where $\a$ is the longest common prefix
of~$\add\TT\aa$ and~$\add\TT\dd$, hereafter denoted $\add\TT\aa
\gcp \add\TT\dd$. This address is also~$\add\TTT\aa
\gcp \add\TTT\dd$.

\subsection{The covering relation}
\label{S:Covering}

The main tool that will enable us to use names to establish
lower bounds for the rotation distance is a binary relation that provides a
description of the shape of a tree in terms of the names of its leaves. 

\begin{defi}
\label{D:Covering}
(See Figure~\ref{F:Covering}.) If $\ii < \jj$ are labels in~$\TT$, we say that
$\ii$ is \emph{covered} by~$\jj$ in~$\TT$, denoted $\ii \cov\TT \jj$, if
there exists a subtree~$\TTT$ of~$\TT$ such that $\ii$ is a non-final label
in~$\TTT$ and $\jj$ is the final label in~$\TTT$.
\end{defi}

\begin{figure}[htb]
\begin{picture}(91,32)(0,0)
\put(0,0){\includegraphics{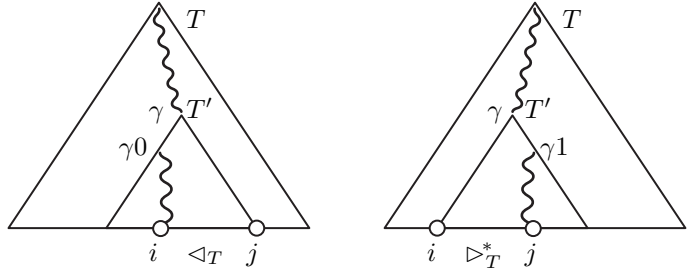}}
\put(25,29){$\TT$}
\put(25,17){$\TTT$}
\put(20,17){$\g$}
\put(16,12){$\g0$}
\put(20,-2){$\ii$}
\put(33,-2){$\jj$}
\put(25,-2){$\cov\TT$}
\put(75,29){$\TT$}
\put(70,17){$\TTT$}
\put(65,17){$\g$}
\put(72,12){$\g1$}
\put(57,-2){$\ii$}
\put(70,-2){$\jj$}
\put(62,-2){$\ccov\TT$}
\end{picture}
\caption{\sf\smaller The covering and co-covering relations: $\ii$ is covered by~$\jj$
in~$\TT$ if there exists a subtree~$\TTT$ such that $\ii$ is a non-final
label in~$\TTT$, whereas is the last (rightmost) label in~$\TTT$. Symmetrically, $\ii$
co-covers~$\jj$ in~$\TT$ if there exists a subtree~$\TTT$ such that $\ii$ is a the initial
(leftmost) label in~$\TTT$, and $\jj$ is a non-initial label in~$\TTT$.}
\label{F:Covering}
\end{figure}

The covering relation~$\cov\TT$ provides a complete description of the
tree~$\TT$.

\begin{prop}
Every labeled tree~$\TT$ is determined by its
covering relation~$\cov\TT$.
\end{prop}

\begin{proof}
Assume first that $\TT$ has the default labeling, \ie, the labels are $1$ to
$\size\TT + 1$. By construction, $\size\TT+1$ is the largest label
in~$\TT$,  and every label~$\lse\size\TT$ is covered by~$\size\TT + 1$,
so $\cov\TT$ determines~$\size\TT$. For an induction, it suffices now to
show that, for $\TT = \TT_0 \op \TT_1$, the relation~$\cov\TT$
determines the size~$\size{\TT_0}$ of~$\TT_0$. Now $1$ is covered
in~$\TT$ by~$\size{\TT_0}+1$, the final label in~$\TT_0$, but it is not
covered by~$\size{\TT_0}+2$, nor is it either covered by any
integer~$\lse\size\TT$. So we have $\size{\TT_0} = \max\{\jj \le
\size\TT  \mid 1 \cov\TT \jj \}$.

The argument is similar for an arbitrary labeled tree, which is determined
by its shape plus the family of its labels.
\end{proof}

According to the notation of~\eqref{E:Subtree}, each subtree of a tree is
specified by a binary address. Introducing the address of
the subtree~$\TTT$ involved in Definition~\ref{D:Covering} 
gives the following rewording of the definition. We recall
that, if $\a$ and~$\b$ are binary addresses, $\a \prefe
\b$ means that $\a$ is a prefix of~$\b$, \ie, $\b$ consists of~$\a$
possibly followed by additional~$0$'s and~$1$'s.

\begin{lemm}
\label{L:CovAddress}
If $\ii, \jj$ are labels in~$\TT$, then $\ii$ is \emph{covered} by~$\jj$ in~$\TT$ if and only if there exists a binary
address~$\g$ satisfying
\begin{equation}
\label{E:Covering} 
\g0 \prefe \add\TT\ii
\text{\quad and \quad}
\add\TT\jj = \g 1^\pp\text{\ for some $\pp \ge 1$.}
\end{equation}
\end{lemm}

For future reference, we mention some simple properties of the covering
relation.

\begin{lemm}
\label{L:CovInterval}
$(i)$ The set of labels covered by~$\jj$ in~$\TT$ is a (possibly empty) interval ending in~$\jjm$: if $\ii \cov\TT
\jj$ and $\ii \le \ii' < \jj$ hold, then so does $\ii' \cov\TT \jj$.

$(ii)$ The relation~$\cov\TT$ is transitive.
\end{lemm}

\begin{proof}
Point $(i)$ is clear from the definition: the elements covered by~$\jj$ in~$\TT$ are the non-final labels in the
maximal subtree~$\TTT$ of~$\TT$ that admits~$\jj$ as its final label: this means that $\TTT$ is either~$\TT$ itself,
or it is some subtree~$\sub\g\TT$ such that $\g$ ends with~$0$.

For~$(ii)$, we observe that, if $\TTT$ is a subtree of~$\TT$ witnessing for $\jj \cov\TT \kk$, and $\TTT'$ is a
subtree witnessing for $\ii \cov\TT \jj$, then the subtrees~$\TTT$ and~$\TTT'$ cannot be disjoint since $\jj$ is the
label of a leaf that belongs to both of them, and, then, $\TTT'$ must be included in~$\TTT$, so
each non-final label in~$\TTT'$ is a non-final label in~$\TTT$ as well. 
\end{proof}

\subsection{The co-covering relation}
\label{S:CoCovering}

The definition of covering gives  a distinguished role to the right side.
Of course, there is a symmetric version involving the left side.

\begin{defi}
\label{D:CoCovering}
(See Figure~\ref{F:Covering}.)
If $\ii < \jj$ are labels in~$\TT$, we say that $\ii$ \emph{co-covers}~$\jj$
in~$\TT$, denoted $\ii \ccov\TT \jj$, if, for some subtree~$\TTT$
of~$\TT$, the integer~$\ii$ is the initial label in~$\TTT$, and $\jj$ is a
non-initial label in~$\TTT$.
\end{defi}

To avoid confusion, we shall always state the covering and co-covering
properties for increasing pairs of labels $\ii < \jj$, thus saying ``$\ii$ is
covered by~$\jj$'' rather than ``$\jj$ covers~$\ii$'', and ``$\ii$
co-covers~$\jj$'' rather than ``$\jj$ is co-covered by~$\ii$''.  The
counterpart of~\eqref{E:Covering} is that $\ii$ co-covers~$\jj$ in~$\TT$ if
and only if there exists~$\g$ satisfying
\begin{equation}
\label{E:CoCovering} 
\add\TT\ii = \g 0^\pp\text{\ for some $\pp \ge 1$}
\text{\quad and \quad}
\g1 \prefe \add\TT\jj.
\end{equation}
The counterparts of Lemmas~\ref{L:CovAddress} and~\ref{L:CovInterval}
are obviously true.

More interesting are the relations that connect covering and co-covering.

\begin{lemm}
\label{L:CovShift}
Assume $\ii < \kk \le \jj < \ell$. Then the relations $\ii \cov\TT \jj$ and $\kk \ccov\TT \ell$ exclude each
other.
\end{lemm}

\begin{proof}
Assume $\ii < \kk \le \jj$, $\ii \cov\TT \jj$ and $\kk \ccov\TT \ell$. We claim that $\ell \le \jj$
holds. Indeed, let $\sub\a\TT$ be a subtree of~$\TT$ witnessing for $\ii \cov\TT \jj$, and $\sub\b\TT$ be a subtree
witnessing for $\kk \ccov\TT \ell$. According to~\eqref{E:Covering} and~\eqref{E:CoCovering}, we have
$$\a0 \prefe \add\II\ii, \quad 
\add\TT\jj = \a1^\pp, \quad
\add\TT\kk = \b0^\qq, \quad\text{and} \quad
\b1 \prefe \add\TT\ell$$
for some~$\pp, \qq \ge 1$. If $\a$ lies on the left of~$\b$, then $\a1^\pp$ lies on the left of~$\b0^\qq$ as well,
contradicting the hypothesis $\kk \le \jj$. If $\a$ lies on the right of~$\b$, then
$\add\TT\ii$ lies on the right of~$\b0^\qq$ as well, contradicting the
hypothesis $\ii < \kk$. So $\a$ and $\b$ must be comparable, \ie, one of
the subtrees~$\sub\a\TT, \sub\b\TT$ is included in the other. The only
case that makes $\kk \le \jj$ possible is $\sub\b\TT$ being strictly
included in~$\sub\a\TT$, in which case we have $\ii < \kk \le \ell \le \jj$.
\end{proof}

\begin{lemm}
\label{L:CCover}
Assume $\aa \cov\TT \bb$.
Then the following are equivalent:

$(ii)$ we have $\add\TT\aa = \g 0 1^\pp$ for some~$\pp \ge 0$,
where $\g$ is~$\add\TT\aa \gcp \add\TT\bb$;

$(ii)$ $\aap \ccov\TT \bb$ holds; 

$(iii)$ $\aa \cov\TT \ii$ fails for each~$\ii$ in~$[\aap, \bbm]$.
\end{lemm}

\begin{proof}
(See Figure~\ref{F:CCover}.)
By Lemma~Ê\ref{L:CovAddress}, the hypothesis implies that
$\add\TT\bb$ is~$\g1^\qq$ for some positive~$\qq$. Assume~$(i)$.
Then
$\add\TT{\aa+1}$ must be of the form~$\g 1 0^\rr$ for some
positive~$\rr$ and, therefore, $\aa+1$ co-covers~$\bb$ in~$\TT$, \ie,
$(ii)$ holds. On the other hand, $\aa$ can be covered by no label
between~$\aa+1$ and~$\bb-1$, so $(iii)$ holds. 

Conversely, we have $\g \prefe \add\TT{\aa+1} \gcp
\add\TT\bb$  under the hypotheses. Hence, if $(ii)$ holds, then
$\add\TT{\aa+1}$ must be of the form~$\g 1 0^\rr$, in which case
$\add\TT\aa$ must be of the form~$\g 0 1^\pp$ for some
positive~$\pp$. Hence, by the counterpart of
Lemma~\ref{L:CovAddress}, $(ii)$ holds.

Finally, if $(i)$ fails, the address of~$\aa$ in~$\TT$ must be of the form
$\g 0 1^\pp 0 \g'$ for some nonnegative~$\pp$ and some~$\g'$. But,
then, there exists a label~$\cc > \aa$ whose address has the form~$\g 0
1^\rr$, and we have $\aa \cov\TT \cc$, so $(iii)$ fails.
\end{proof}

\begin{figure}[htb]
\begin{picture}(70,26)(0,-1)
\put(0,0){\includegraphics{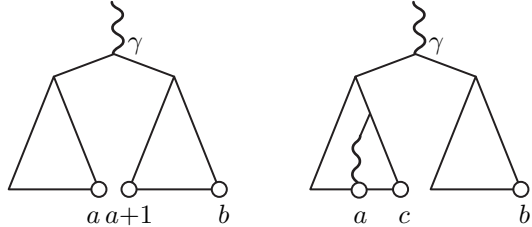}}
\put(10.5,-3){$\aa$}
\put(13,-3){$\aa{+}1$}
\put(28,-3){$\bb$}
\put(16,20){$\g$}
\put(56,20){$\g$}
\put(46,-3){$\aa$}
\put(52,-3){$\cc$}
\put(68,-3){$\bb$}
\end{picture}
\caption{\sf\smaller Proof of Lemma~\ref{L:CCover}: on the left, the case 
when $(i)$, $(ii)$, $(iii)$ hold; on the right, the case when they fail.}
\label{F:CCover}
\end{figure}

\subsection{The Key Lemma}
\label{S:Principle}

We arrive at the main point, namely analyzing the influence of
rotations on the covering and co-covering relations. The result is
simple: a positive rotation creates some covering and deletes some
co-covering, a negative rotation does the contrary. The nice
point is that there exists a close relation between the name of a
rotation and the covering or co-covering pairs it creates or deletes. This
will directly lead to lower bounds on the rotation distance: as one rotation
only changes the covering and co-covering relations by a small, well
controlled amount, if two trees~$\TT,
\TTT$ have very different covering and co-covering relations, many
rotations are needed to transform~$\TT$ into~$\TTT$.

First, we note for further reference the following straighforward facts,
whose verification should be obvious from Figure~\ref{F:Name}. We
naturally use $\ii \cove\TT \jj$ for ``$\ii \cov\TT \jj$ or $\ii = \jj$'', and
similarly for $\ii \ccove\TT \jj$.

\begin{lemm}
\label{L:NameCov}
If $(\TT, \TTT)$ is a base pair with name $(\aa, \bb, \cc,
\dd)^\pm$, we have

$(i)$ $\aa \ccove\TT \ii$ and $\ii \cove\TT \dd$ for each label~$\ii$
of~$\TT$ lying in~$[\aa, \dd]$;

$(ii)$ $\aa \ccove\TT \ii$ and $\ii \cove\TT \bb-1$ for each label~$\ii$
of~$\TT$ lying in~$[\aa, \bb-1]$;

$(iii)$ $\bb \ccove\TT \ii$ and $\ii \cove\TT \cc$ for each label~$\ii$
of~$\TT$ lying in~$[\bb, \cc]$;

$(iv)$ $\cc+1 \ccove\TT \ii$ and $\ii \cove\TT \dd$ for each label~$\ii$
of~$\TT$ lying in~$[\cc+1, \dd]$.
\end{lemm}

The roles of~$\TT$ and~$\TTT$ in Lemma~\ref{L:NameCov} are
symmetric, so all covering results stated for~$\TT$ also hold for~$\TTT$.

Here comes the main point, namely the way covering and anticovering
change in a base pair.

\begin{lemm}
\label{L:CovChange}
If $(\TT, \TTT)$ is a base pair with name $(\aa, \bb, \cc,
\dd)^+$, then, for all labels~$\ii, \jj$ occurring in~$\TT$, 

$(i)$ $\ii \cov\TTT \jj$ holds if and only if we have either $\ii \cov\TT \jj$, or $\aa \le \ii <
\bb$ and $\jj = \cc$; 

$(ii)$ $\ii \ccov\TT \jj$ holds if and only if we have either $\ii \ccov\TTT \jj$, or $\ii = \bb$ and $\cc < \jj
\le \dd$.
\end{lemm}

\begin{proof}
$(i)$ For~$\ii < \aa$, we have $\ii \not\cov\TT \jj$ and $\ii \not\cov\TTT \jj$ for~$\jj < \dd$, and
$\ii \cov\TT \jj \Longleftrightarrow\ii \cov\TTT \jj$ for~$\jj \ge \dd$. Similarly, we have $\ii
\cov\TT \jj \Longleftrightarrow\ii \cov\TTT \jj$ for $\dd < \ii < \jj$. So the point is to consider the
pairs~$(\ii, \jj)$ with $\aa \le \ii < \jj \le \dd$. As is clear from Figure~\ref{F:Name}, nothing changes
from~$\TT$ to~$\TTT$ when $\ii$ and~$\jj$ both are in $[\aa, \bb-1]$, or in $[\bb, \cc]$, or in $[\cc+1, \dd]$.
For the cases when $\ii$ lies in one of the intervals and $\jj$ in another,
one sees on the figure that the only case when $\cov\TT$ and $\cov\TTT$
disagree is when $\ii$ lies in~$[\aa, \bb-1]$, thus corresponding to a leaf
in the left subtree (the one with address~$\a0$), and $\jj$ is corresponds
to the rightmost leaf in the central tree, which has address~$\a10$
in~$\TT$ and~$\a01$ in~$\TTT$. In this case, we have $\ii \not\cov\TT
\jj$ and $\ii \cov\TTT \jj$.

The case of~$(ii)$ is symmetric, exchanging the roles of left and right and
the inequalities.
\end{proof}

So, if $(\TT, \TTT)$ is a positive base pair, the only difference between
the relations~$\cov\TT$ and~$\cov\TTT$ are that $\Name3\TT\TTT$
covers more elements in~$\TTT$ than in~$\TT$. Symmetrically, the
relations~$\ccov\TT$ and $\ccov\TTT$ are equally close, the only
difference being that $\Name2\TT\TTT$ co-covers less elements in~$\TTT$
than in~$\TT$. This leads us to a criterion for recognizing that
certain types of base pairs inevitably occur on any path connecting two
trees. 

\begin{nota}
In the sequel, we say ``pair $(..., ..., ..., ...)^\pm$'' for 
``base pair with name $(..., ..., ..., ...)^\pm$''. Also, we use abbreviated
notation such as $\Qp(\lse\ii, \gs\ii, \jj, ...)$ to refer to any pair
$\Qp(\aa, \bb, \cc, \dd)$ satisfying $\aa \le \ii$, $\bb > \ii$, and $\cc =
\bb$.
\end{nota}

\begin{lemm}[{\bf Key Lemma}]
\label{L:KeyLemma1}
Assume that the trees $\TT, \TTT$ satisfy
\begin{equation}
\label{E:Principle}
\ii \not\cov\TT \jj
\text{\quad and \quad}
\ii \cov\TTT \jj.
\end{equation}
Then each sequence of rotations from~$\TT$ to~$\TTT$ contains a pair~$\Qp(\lse\ii, \gs\ii, \jj, ...)$.
\end{lemm}

\begin{proof}
Let $(\TT_0, \Ldots, \TT_\ell)$ be a path from~$\TT$ to~$\TTT$
in~$\Ass{\size\TT}$, \ie, assume $\TT_0 = \TT$, $\TT_\ell = \TTT$, and
$(\TT_\rr, \TT_\rrp)$ is a base pair for each~$\rr$.
By~\eqref{E:Principle}, we have
$\ii \not\cov{\TT_0} \jj$ and $\ii \cov{\TT_\ell} \jj$, so there must exist an integer~$\rr$ satisfying
\begin{equation}
\label{E:Principle2}
\ii \not\cov{\TT_\rr} \jj
\text{\quad and \quad}
\ii \cov{\TT_\rrp} \jj.
\end{equation}
By Lemma~\ref{L:CovChange}$(i)$, this can occurs only if we have
$$\Name1{\TT_\rr}{\TT_\rrp} \le \ii < \Name2{\TT_\rr}{\TT_\rrp}
\text{\quad and \quad}
\Name3{\TT_\rr}{\TT_\rrp} = \jj.
$$
In other words, the pair~$(\TT_\rr, \TT_\rrp)$ has the form~$\Qp(\lse\ii, \gs\ii, \jj, ...)$.
\end{proof}

Thus the covering relations partition the associahedra into regions.
What Lem\-ma~\ref{L:KeyLemma1} says is that one cannot go from one
region to another one without crossing the border, which corresponds to
base pairs with a certain type of name, see Figure~\ref{F:K4Names}. 

\begin{figure}[htb]
\begin{picture}(88,67)(0,0)
\put(-1,-1){\includegraphics[scale=0.8]{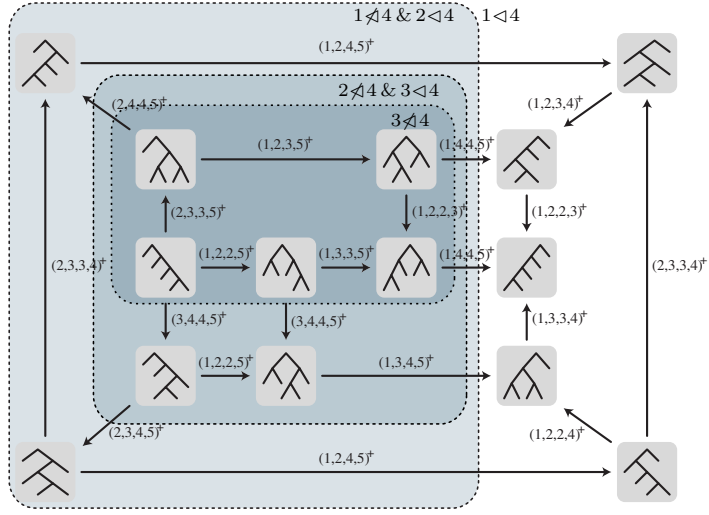}}
\put(62,64){$\scriptstyle 1 \cov{} 4$}
\put(45,64){$\scriptstyle 1 \not\cov{} 4 \, \&\,  2 \cov{} 4$}
\put(43,54){$\scriptstyle 2 \not\cov{} 4 \, \&\,  3 \cov{} 4$}
\put(50,50){$\scriptstyle 3 \not\cov{} 4$}
\end{picture}
\caption{\sf\smaller Names of edges and covering relation in the
associahedron~$\Ass4$. The regions correspond to the various
possibilities for covering by~$4$. As
$\ii
\cov{}4$ implies
$\jj \cov{} 4$ for $\ii < \jj < 4$, there are four regions corresponding to $1
\cov{} 4$, to $1 \not\cov{} 4$ and $2
\cov{} 4$, to $2 \not\cov{} 4$ and $3 \cov{} 4$, and to $3 \not\cov{} 4$.
Lemma~\ref{L:KeyLemma1} says for instance that the only way to leave the
region $3 \not\cov{} 4$ is to cross a pair named $\Qp(..., 4, 4, ...)$.}
\label{F:K4Names}
\end{figure}

Of course, we have a symmetric statement involving the relation~$\ccov{}$.

\begin{lemm}
\label{L:KeyLemma2}
Assume that the trees $\TT, \TTT$ satisfy
\begin{gather}
\label{E:PrincipleBis}
\ii \ccov\TT \jj
\text{\quad and \quad}
\ii \not\ccov\TTT \jj,\\
\label{E:PrincipleTer}
\text{or \quad}
\ii \ccov\TT \jj
\text{\quad and \quad}
\iim \cov\TTT \jjm.
\end{gather}
Then each sequence of rotations from~$\TT$ to~$\TTT$ contains a
pair~$\Qp(..., \ii, \ls\jj,
\gse\jj)$.
\end{lemm}

\begin{proof}
The statement for~\eqref{E:PrincipleBis} is the exact conterpart of Lemma~\ref{E:Principle} when left and right are
exchanged. As for~\eqref{E:PrincipleTer}, Lemma~\ref{L:CovShift} says that $\iim \cov\TTT \jjm$ implies
$\ii \not\ccov\TTT \jj$, and, therefore, the hypotheses~\eqref{E:PrincipleTer} imply the
hypotheses~\eqref{E:PrincipleBis}.
\end{proof}

\subsection{Refinements}
\label{S:Refinements}

More precise criteria will be needed in the sequel, and we shall now establish some refinements of
Lemmas~\ref{L:KeyLemma1} and~\ref{L:KeyLemma2}. All are based on
these basic results, but, in addition, they exploit the geometric properties
of the relations~$\cov{}$ and~$\ccov{}$. The crucial advantage of the 
criteria below is that they provide stronger constraints for the
parameters of the involved pairs: for instance, Lemma~\ref{L:PPrinciple}
specifies two of the four parameters completely (``$\iip$ and $\jj$''),
whereas Lemma~\ref{L:KeyLemma1} only gives one exact value (``$\jj$''),
and not more than an inequality for another one (``$\gs\ii$''). The price
to pay for the improvement is a strengthtening of the hypotheses and,
chiefly, a disjunction in the conclusion.

\begin{lemm}
\label{L:PPrinciple}
Assume that the trees $\TT, \TTT$ satisfy
\begin{equation}
\label{E:PPrinciple}
\ii \not\cov\TT \jj, 
\quad
\ii \cov\TTT \jj,
\text{\quad and \quad}
\iip \ccove\TTT \jj.
\end{equation}
Then each sequence of rotations from~$\TT$ to~$\TTT$ contains a
pair~$\Qp(..., \iip, \jj, ...)$, or a pair~$\Qm(..., \iip, ..., \jj)$.
\end{lemm}

Here we shall prove Lemma~\ref{L:PPrinciple} directly---this can be done
as a good exercise---but rather derive it from a more elaborate statement.
The new refinement consists in getting constraints on three parameters at a
time.

\begin{lemm}
\label{L:PPPrinciple}
Assume $\ii < \kk < \jj$ and the trees $\TT, \TTT$ satisfy
\begin{equation}
\label{E:PPPrinciple}
\ii \not\cov\TT \jj, 
\quad
\ii \cov\TTT \jj,
\text{\quad and \quad}
\iip \ccove\TTT \jj.
\end{equation}
Then each sequence of rotations from~$\TT$ to~$\TTT$ contains 
a pair~$\Qp(\lse\ii, \ii\ls...\lse\kk, \jj, ...)$, or a pair~$\Qm(\lse\ii, \ii\ls...\lse\kk, \gse\kk, \jj)$.
\end{lemm}

\begin{proof}[Proof of Lemma~\ref{L:PPrinciple} from Lemma~\ref{L:PPPrinciple}]
Assume first $\jj > \iip$. Put $\kk = \iip$. Then $\ii <  \bb \le\kk$ implies 
$\bb = \iip$. So Lemma~\ref{L:PPPrinciple} guarantees that there
is at least one pair named
$$\Qp(\lse\ii, \iip, \jj, ...)
\mbox{\quad or \quad}
\Qm(\lse\ii, \iip, \gse\iip, \jj),$$
which is the expected conclusion. 

Assume now $\jj = \iip$. By hypothesis, we have $\ii \not\cov\TT \jj$ and $\ii \cov\TTT \jj$, so
Lemma~\ref{L:KeyLemma1} gives a pair $\Qp(\lse\ii, \gs\ii, \jj, ...)$, hence necessarily
$\Qp(..., \iip, \iip, ...)$, again of the expected type.
\end{proof}

\begin{proof}[Proof of Lemma~\ref{L:PPPrinciple}]
We begin as in the proof of Lemma~\ref{L:KeyLemma1}.
Let $(\TT_0, \Ldots, \TT_\ell)$ be a path from~$\TT$ to~$\TTT$
in~$\Ass{\size\TT}$. By~\eqref{E:PPPrinciple}, we have
$\ii \not\cov{\TT_0} \jj$ and $\ii \cov{\TT_\ell} \jj$, so there exists a largest integer~$\rr$ satisfying
\begin{equation}
\label{E:PPPrinciple2}
\ii \not\cov{\TT_\rr} \jj
\text{\quad and \quad}
\ii \cov{\TT_\rrp} \jj,
\end{equation}
and, as above, this requires
$\Name1{\TT_\rr}{\TT_\rrp} \le \ii < \Name2{\TT_\rr}{\TT_\rrp}$ and $\Name3{\TT_\rr}{\TT_\rrp} = \jj$, so $(\TT_\rr,
\TT_\rrp)$ is a pair~$\Qp(\lse\ii, \gs\ii, \jj, ...)$. Write
$\bb = \Name2{\TT_\rr}{\TT_\rrp}$. 

If $\bb \le \kk$ holds, $(\TT_\rr, \TT_\rrp)$ is a pair~$\Qp(\lse\ii,
\kk\lse...\ls\ii, \jj, ...)$, and we are done.

So, from now on, we assume $\kk < \bb$.  The hypotheses $\ii < \kk$
implies $\ii < \bbm$, hence $\ii \cov{\TT_\rr} \bbm$ by
Lemma~\ref{L:NameCov}$(i)$. As we have
$\bbm < \bb  \le \Name3{\TT_\rr}{\TT_\rrp} = \jj$, we deduce
\begin{equation}
\label{E:PPPrinciple3}
\exists\xx \in [\kk, \jjm] (\ii \cov{\TT_\rr} \xx)
\end{equation}
On the other hand, by Lemma~\ref{L:CCover}, the hypothesis $\iip \ccove\TTT \jj$ implies
\begin{equation}
\label{E:PPPrinciple4}
\forall\xx \in [\kk, \jjm] (\ii \not\cov{\TT_\ell} \xx).
\end{equation}
Therefore, there must exist~$\ss \ge \rr$ satisfying
\begin{equation}
\label{E:PPPrinciple5}
\exists\xx \in [\kk, \jjm] (\ii \cov{\TT_\ss} \xx)
\text{\quad and \quad}
\forall\xx \in [\kk, \jjm] (\ii \not\cov{\TT_\ssp} \xx).
\end{equation}
Choose such a~$\ss$. Then, for some~$\ee$ in~$[\kk, \jjm]$, we have 
\begin{equation}
\label{E:PPPrinciple6}
\ii \cov{\TT_\ss} \ee
\text{\quad and \quad}
\ii \not\cov{\TT_\ssp} \ee.
\end{equation}
By Lemma~\ref{L:CovChange}$(i)$, the name of~$(\TT_\ss, \TT_\ssp)$ has the form
$\Qm(\lse\ii, \gs\ii, \ee, ...)$, hence, a fortiori, $\Qm(\lse\ii, \gs\ii, \gse\kk, ...)$.

Moreover, we have $\ss \ge \rr$ by construction, so the choice of~$\rr$ implies that $\ii$ is covered by~$\jj$
in~$\TT_\ssp$, and, therefore,  by Lemma~\ref{L:CovInterval}, so is $\eep$  since we have $\ii < \kk \le \ee < \jj$.
On the other hand, Lemma~\ref{L:NameCov}$(iv)$ gives $\eep \ccov{\TT_\ssp} \Name4{\TT_\ss}{\TT_\ssp}$, and
Lemma~\ref{L:CCover} then implies that $\eep$ is covered in~$\TT_\ssp$ by no element smaller than
$\Name4{\TT_\ss}{\TT_\ssp}$. We deduce $\Name4{\TT_\ss}{\TT_\ssp} \le \jj$, and, even, 
$\Name4{\TT_\ss}{\TT_\ssp} \in [\kkp, \jj]$ as we have $\Name4{\TT_\ss}{\TT_\ssp} > \Name3{\TT_\ss}{\TT_\ssp} \ge \kk$.
Now, by Lemma~\ref{L:NameCov}$(i)$, $\ii$ is covered by~$\Name4{\TT_\ss}{\TT_\ssp}$ in~$\TT_\ssp$, 
whereas, by~\eqref{E:PPPrinciple5}, it is covered by no element of~$[\kk,
\jjm]$. It follows that the only possibility is 
$\Name4{\TT_\ss}{\TT_\ssp} = \jj$.

Finally, let $\ff = \Name2{\TT_\ss}{\TT_\ssp}$. We already know that $\ff > \ii$ holds, and we claim that $\ff \le \kk$
holds as well. Indeed, two cases are possible. For $\ff = \iip$, the hypothesis $\ii < \kk$ directly implies $\ff \le
\kk$. For $\ff > \iip$, Lemma~\ref{L:NameCov}$(ii)$ implies that $\ii$ is covered by~$\ffm$ in~$\TT_\ssp$, and,
therefore, \eqref{E:PPPrinciple5} implies that $\ffm$ cannot belong to~$[\kk, \jjm]$. As $\ff \le \ee \le \jjm$ is
true by construction, the only possibility is $\ffm < \kk$, \ie, $\ff \le \kk$. 

So $(\TT_\ss, \TT_\ssp)$ is a pair~$\Qm(\lse\ii, \ii\ls...\lse\kk, \gse\kk, \jj)$, as expected.  
\end{proof}

For future reference, we finally mention the right counterpart of
Lemma~\ref{L:PPrinciple}. It can of course be proved by a direct argument,
or deduced from the right counterpart of Lemma~\ref{L:PPPrinciple} (that
we shall not need here).

\begin{lemm}
\label{L:PPrincipleBis}
Assume that the trees $\TT, \TTT$ satisfy
\begin{equation}
\label{E:PPrinciple}
\ii \ccov\TT \jj, 
\quad
\ii \cove\TT \jjm,
\text{\quad and \quad}
\ii \not\ccov\TTT \jj.
\end{equation}
Then each sequence of rotations from~$\TT$ to~$\TTT$ contains a pair~$\Qp(..., \ii, \jjm,
...)$, or a pair~$\Qm(\ii, ..., \jjm, ...)$.
\end{lemm}

\subsection{Application: reproving a lower bound in $3\nn/2 + O(1)$}
\label{S:Bicomb}

As a first application of the previous results and a warm-up for the sequel, we shall now reprove
Proposition~\ref{P:Bicomb} about the distance of ``bicombs''.

\begin{prop}
\label{P:BicombBis}
For $\TT = \Sp{1^\pp0^\qq}$, $\TTT = \Sp{0^\qq1^\pp}$ with~$\pp, \qq \ge 1$, we have
\begin{equation}
\dist(\TT, \TTT) \ge \pp + \qq + \min(\pp, \qq) - 2.
\end{equation}
\end{prop}

\begin{figure}[htb]
\begin{picture}(67,33)(0,-2)
\put(0,-0.5){\includegraphics{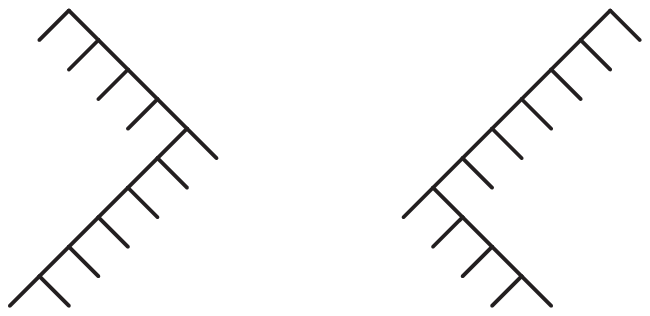}}
\put(13,29){$\TT$}
\put(2,25){$\scriptstyle 1$}
\put(11,16){$\scriptstyle \pp$}
\put(-2,-2){$\scriptstyle \pp\plus1$}
\put(6,-2){$\scriptstyle \pp\plus2$}
\put(18,10){$\scriptstyle \nn\, (=\pp\plus\qq)$}
\put(21,13){$\scriptstyle \nn\plus1$}

\put(54,29){$\TTT$}
\put(39,7){$\scriptstyle 1$}
\put(48,-2){$\scriptstyle \pp$}
\put(56,-2){$\scriptstyle \pp\plus1$}
\put(50,10){$\scriptstyle \pp\plus2$}
\put(62,22){$\scriptstyle \nn$}
\put(65,25){$\scriptstyle \nn\plus1$}
\end{picture}
\caption{\sf\smaller The bicombs of Propositions~\ref{P:Bicomb}
and~\ref{P:BicombBis}---here with
$\pp = 4$ and $\qq = 6$.}
\label{F:Bicomb}
\end{figure}

The method of the proof consists in identifying various families of base
pairs, and to prove, using the results of Sections~\ref{S:Principle}
and~\ref{S:Refinements}, that every sequence of rotations from~$\TT$
to~$\TTT$ contains at least a certain number of pairs of these specific
types.

\begin{proof}
Up to a symmetry, we may assume $\pp \le \qq$. Put $\nn = \pp + \qq$. Let us say that a base pair is
\emph{special}
\begin{tabbing}
\ 
\=- \emph{of type $\tI_\aa$} \hspace{0.0em}
\=if it is $\Qp(..., \aa, \pp\plus1, ...)$ or $\Qm(..., \aa, ..., \pp\plus1)$ 
\=with $2 \le \aa \le \ppp$,\\
\>- \emph{of type $\tII_\aa$} 
\>if it is $\Qp(..., \pp\plus1, \aa, ...)$ or $\Qm(\pp\plus1, ..., \aa, ...)$
\>with $\pp\plus2 \le \aa \le \nn$,\\
\>- \emph{of type $\tIII_\aa$}
\>if it is $\Qp(1, {\not=}\pp\plus1, \aa, ...)$
\>with $\pp\plus2 \le \aa \le \nn$,\\
\>- \emph{of type $\tIV_\aa$}
\>if it is $\Qp(..., \aa, \lse\pp, ...)$
\>with $2 \le \aa \le \pp$.
\end{tabbing}
First we observe that the various types of special pairs are disjoint, \ie, a special pair 
has one type exactly: the four families are disjoint and, inside each family,
the parameter~$\aa$ is uniquely determined.

Let ($\TT_0, ..., \TT_\ell)$ be a path from~$\TT$ to~$\TTT$
in~$\Ass\nn$. First choose~$\aa$ in~$[2, \ppp]$. We see on
Figure~\ref{F:Bicomb} that $\ppp$ covers~$\aam$ in~$\TTT$, but not
in~$\TT$, and, moreover, that $\aa$ co-covers~$\ppp$ in~$\TT$.
Applying Lemma~\ref{L:PPrinciple} with $\ii = \aam$ and~$\jj = \ppp$
guarantees that $(\TT_0, ..., \TT_\ell)$ contains a pair $\Qp(..., \aa,
\ppp, ...)$ or $\Qm(..., \aa, ..., \ppp)$, \ie, a special pair of
type~$\tI_\aa$. Letting $\aa$ vary from~$2$ to~$\ppp$ guarantees that
$(\TT_0, ..., \TT_\ell)$ contains at least $\pp$~special pairs of type~$\tI$.

Similarly, choose~$\bb$ in~$[\pp+2, \nn]$. We see now that $\ppp$ 
co-covers~$\bb$ in~$\TT$, but not in~$\TTT$, and, that, moreover,
$\ppp$ is not covered by~$\bbp$ in~$\TTT$. Applying
Lemma~\ref{L:PPrincipleBis} with $\ii = \ppp$ and~$\jj = \bbp$
guarantees that $(\TT_0, ..., \TT_\ell)$ contains a pair $\Qp(..., \ppp,
\bb, ...)$ or $\Qm(\ppp, ..., \bb, ...)$, \ie, a special pair of
type~$\tII_\bb$. Letting $\bb$ vary from~$\pp+2$ to~$\nn$ guarantees
that $(\TT_0, ..., \TT_\ell)$ contains at least
$\qqm$~special pairs of type~$\tII$.

Consider~$\bb$ in~$[\pp+2, \nn]$ again. Then $1$ is covered by~$\ppp$ in~$\TT$, and not covered
by~$\ppp$ in~$\TTT$. Applying Lemma~\ref{L:KeyLemma1} with~$\ii = 1$ and~$\jj = \bb$ guarantees that $(\TT_0,
..., \TT_\ell)$ contains at least one pair $\Qp(1, \aa, \bb, ...)$ for some~$\aa$ satisfying $2 \le \aa \le \bb$. Here
two cases are possible.

\medskip
{\bf Case 1}. For each~$\bb$ in~$[\pp+2, \nn]$, there is a pair $\Qp(1, \aa, \bb, ...)$ with $\aa \not= \ppp$, or
type~$\tIII_\a$. In this case, letting $\bb$ vary from~$\pp+2$ to~$\nn$
guarantees that $(\TT_0, ..., \TT_\ell)$ contains at least $\qqm$~special
pairs of type~$\tIII$.

\medskip
{\bf Case 2}. There exists~$\bb$ such that $(\TT_0, ..., \TT_\ell)$ contains no special pair of
type~$\tIII_\bb$. Owing to the above observation, this implies that there exists~$\rr$ such that $(\TT_\rr,
\TT_\rrp)$ is $\Qp(1, \ppp, \bb, ...)$. By Lemma~\ref{L:NameCov}$(ii)$, this implies that $1$ is covered by~$\pp$
in~$\TT_\rr$. In this case, we claim that there must exist in $(\TT_0, ..., \TT_\rr)$ a
pair of type~$\tIV_\aa$ for each~$\aa$ in~$[2, \pp]$. Indeed, consider
such an~$\aa$. The hypothesis that $1$ is covered by~$\pp$
in~$\TT_\rr$ implies that $\aam$ too is covered by~$\pp$ in~$\TT_\rr$.
On the other hand, we see that $\aa$ co-covers~$\ppp$ in~$\TT$.
Applying Lemma~\ref{L:KeyLemma2} with~$\ii = \aa$ and~$\jj = \ppp$
guarantees that $(\TT_0, ..., \TT_\rr)$ contains a pair  $\Qp(..., \aa,
\lse\pp,
\gs\pp)$, hence a special pair of type~$\tIV_\aa$. So, in this case,
$(\TT_0, ..., \TT_\ell)$ contains at least $\ppm$~special pairs of
type~$\tIV$.

Summarizing, we conclude that $(\TT_0, ..., \TT_\ell)$ contains at least 
\begin{tabbing}
\hspace{3em}
\=\quad$\pp$\quad
\=special pairs of type~$\tI$,\\
\>$\qqm$ 
\>special pairs of type~$\tII$,\\
and
\>$\qqm$ \>special pairs of type~$\tIII$ or $\ppm$ special pairs of type~$\tIV$,
\end{tabbing}
hence at least $3\pp-2$ special pairs. As these pairs are pairwise distinct, the distance between~$\TT_0$
and~$\TT_\ell$, \ie, between~$\TT$ and~$\TTT$, is at least $3\pp-2$. 
\end{proof}

The previous argument is illustrated in the case of size~$4$
trees in Figure~\ref{F:K4Colours}.

\begin{figure}[htb]
\begin{picture}(88,62)(0,0)
\put(-1,-0.5){\includegraphics[scale=0.8]{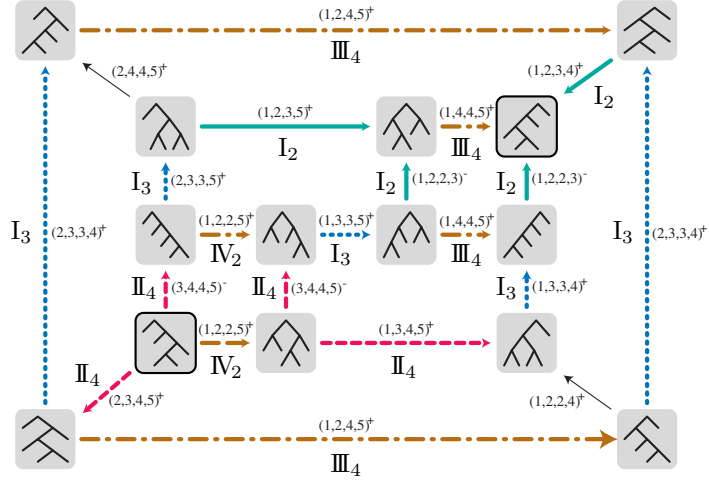}}
\put(42,0){$\tIII_4$}
\put(42,28){$\tI_3$}
\put(42,55){$\tIII_4$}
\put(26,28){$\tIV_2$}
\put(58,28){$\tIII_4$}
\put(58,42){$\tIII_4$}
\put(35,42){$\tI_2$}
\put(26,13){$\tIV_2$}
\put(50,13){$\tII_4$}
\put(15.5,37.5){$\tI_3$}
\put(48,37.5){$\tI_2$}
\put(64,37.5){$\tI_2$}
\put(15.5,23.5){$\tII_4$}
\put(31.5,23.5){$\tII_4$}
\put(64,23.5){$\tI_3$}
\put(-0.5,31){$\tI_3$}
\put(79.5,31){$\tI_3$}
\put(8,12){$\tII_4$}
\put(77,49){$\tI_2$}
\end{picture}
\caption{\sf\smaller Proof of Proposition~\ref{P:BicombBis}. Introducing
special pairs of various types amounts  to colouring the edges of the
associahedron~$\Ass4$. In the current case,  we use four colours, namely
$\tI_2$, $\tI_3$, $\tII_4$, and a common colour for~$\tIII_4$
and~$\tIV_2$, plus a neutral colour for the edges that receive none of the
previous colours, \ie, for non-special pairs. The proof shows that each path 
from~$\Sp{1100}$ to~$\Sp{0011}$ must contain at least one edge of each
of the four colours---as can be checked on the picture---hence the distance
between the trees~$\Sp{1100}$ and~$\Sp{0011}$ (framed squares) is
four.}
\label{F:K4Colours}
\end{figure}

\section{Collapsing}
\label{S:Collapsing}

The previous method is powerful, but the results obtained so far remain
limited. In order to establish stronger results, we now add one more
ingredient called collapsing, which is a certain way of projecting an
associahedron~$\Ass\nn$ onto smaller associahedra~$\Ass{\nn'}$ with
$\nn' < \nn$. The idea is very simple: collapsing a set of labels~$\II$ in a
tree~$\TT$ means erasing all leaves whose labels belong to~$\II$, and
contracting the remaining edges to obtain a well-formed
tree~$\coll\II(\TT)$. Collapsing a set of labels~$\II$ in the
two entries of a base pair~$(\TT,
\TTT)$ yields either a base pair, or twice the same tree, in which case we
naturally say that the pair $(\TT,
\TTT)$ is $\II$-collapsing. This implies that the rotation distance
between~$\TT$ and~$\TTT$ is at least the distance between the
collapsed trees $\coll\II(\TT)$ and~$\coll\II(\TTT)$, plus the minimal
number of inevitable $\II$-collapsing pairs between~$\TT$ and~$\TTT$.
This principle will enable us to inductively determine the
distances~$\dist(\TT_\pp,
\TTT_\pp)$ for trees~$\TT_\pp, \TTT_\pp$ such that
$\TT_\ppm$ is obtained from~$\TT_\pp$ and $\TTT_\ppm$ is obtained
from~$\TTT_\pp$ by collapsing some set of labels (the same for both).

\subsection{Collapsing}

Assume that $\TT$ is a finite binary tree. For $\II \ince \Nat$, we consider the tree obtained from~$\TT$ by
removing all leaves whose labels lie in~$\II$ (if any). In order to include the case when all leaves are removed, we
introduce an empty tree denoted~$\et$, together with the rules
\begin{equation}
\label{E:Empty}
\et \op \TT = \TT, \quad
\TT \op \et = \TT.
\end{equation}
It is then coherent to declare~$\size\et = -1$. Objects that are either a
finite labeled tree or the empty tree will be called \emph{extended trees}.
We denote by~$\Lab\TT$ the family of labels occurring in~$\TT$.

\begin{defi}
For $\II\ince \Nat$ and~$\TT$ an extended tree, the
\emph{$\II$-collapse} of~$\TT$, denoted~$\coll\II(\TT)$, is recursively
defined by
\begin{equation}
\coll\II(\TT) = 
\begin{cases}
\et
&\text{for $\TT = \et$},\\
\et
&\text{for $\size\TT=0$ and $\Lab\TT \ince \II$},\\
\TT
&\text{for $\size\TT=0$ and $\Lab\TT \not\ince \II$},\\
\coll\II(\TT_1) \mbox{\rlap{$\op \coll\II(\TT_2)$ 
\quad for $\TT = \TT_1 \op \TT_2$.}}
\end{cases}
\end{equation}
\end{defi}

\begin{exam}
Assume $\TT = ((1 \op 2) \op (3 \op 4))$. Then we find for instance
$$\coll{\{1\}}(\TT) = 2 \op (3 \op 4), \quad
\coll{\{2, 3\}}(\TT) = 1 \op 4, \quad
\coll{\{1,2, 3,4\}}(\TT) = \et.
$$
\end{exam}

Properties of collapsing are mostly obvious. In particular, it should be 
clear that we always have $\size{\coll\II(\TT)} = \size\TT - \card(\Lab\TT
\cap \II)$. Also, we have the following compatibility with the covering and
co-covering relations.

\begin{lemm}
\label{L:CollCov}
If $\ii$ is covered by (\resp co-covers)~$\jj$ in~$\TT$, and $\ii, \jj$ do
not belong to~$\II$, then $\ii$ is covered by (\resp co-covers)~$\jj$
in~$\coll\II(\TT)$.
\end{lemm}

\begin{rema}
We took the option not to change the remaining labels when some labels
are collapsed. So, even we start with a tree~$\TT$ in which the labels are
the default ones, namely~$1$ to~$\size\TT+1$, after collapsing we are
likely to obtain a tree~$\TTb$ in which the labels are not $1$
to~$\size\TTb + 1$. That is why it seems preferable to consider general
labeled trees, \ie, to allow jumps in the sequence of labels.
\end{rema}

\subsection{Collapsing a base pair}

For our current purpose, the question is to connect the
rotation distance between two trees and the rotation distance between their
images under collapsing. The point is that the collapsing of a
base pair is either a diagonal pair, \ie, a pair consisting of twice the same
tree, or it is still a base pair, whose name is easily connected with that of
the initial pair.

\begin{lemm}
\label{L:CollPair}
For $\II \ince \Nat$ and $(\TT, \TTT)$ a base pair of name  $\Qp(\aa, \bb,
\cc, \dd)$, 

$(i)$ either we have 
\begin{equation}
\label{E:CollPair}
[\aa, \bbm] \ince \II
\text{\quad or \quad}
[\bb, \cc] \ince \II
\text{\quad or \quad}
[\ccp, \dd] \ince \II,
\end{equation}
and then $\coll\II(\TT)$ and $\coll\II(\TTT)$ coincide,

$(ii)$ or $(\coll\II(\TT), \coll\II(\TTT))$ is a base pair of name $\Qp(\aab, \bbb, \ccb, \ddb)$ with
\break
$\aab = \min([\aa, \bbm] \setminus \II), 
\bbb = \min([\bb, \cc] \setminus \II), 
\ccb = \max([\bb, \cc] \setminus \II), 
\ddb = \max([\ccp, \dd] \setminus \II).$
\end{lemm}

\begin{proof}
See Figure~\ref{F:Name}. If one of the three subtrees~$\sub{\a0}\TT$,
$\sub{\a10}\TT$, $\sub{\a11}\TT$ completely vanishes, which
happens when at least one of the three inclusions of~\eqref{E:CollPair} is
true, then $\coll\II(\TT)$ and~$\coll\II(\TTT)$ coincide. Otherwise, \ie,
if at least one leaf of each of the above three subtrees remains, we have
\begin{gather*}
\coll\II{(\sub\a\TT)} = 
\coll\II{(\sub{\a0}\TT)} \op (\coll\II{(\sub{\a10}\TT)} \op
\coll\II{(\sub{\a11}\TT)}),\\
\coll\II{(\sub\a\TTT)} = 
(\coll\II{(\sub{\a0}\TT)} \op \coll\II{(\sub{\a10}\TT)}) \op
\coll\II{(\sub{\a11}\TT)},
\end{gather*}
Hence $(\coll\II{(\sub\a\TT)}, \coll\II{(\sub\a\TTT)})$ is a positive base
pair, and so is $(\coll\II(\TT), \coll\II(\TTT))$. The name should then be
clear from the picture
\end{proof}

\begin{defi}
A base pair is called \emph{$\II$-collapsing} if at least one of the three conditions of~\eqref{E:CollPair} is satisfied.
If $\TT, \TTT$ are trees, the \emph{$\II$-distance} between~$\TT$ and~$\TTT$, denoted
$\Dist\II(\TT, \TTT)$, is defined to be the minimal number of $\II$-collapsing steps occurring 
in a sequence of rotations from~$\TT$ to~$\TTT$.
\end{defi}

A direct application of Lemma~\ref{L:CollPair} is the following useful relation:

\begin{lemm}
\label{L:DistColl}
For all trees~$\TT, \TTT$ and all sets~$\II$, we have
\begin{equation}
\label{E:DistColl}
\dist(\TT, \TTT) \ge \dist(\coll\II(\TT), \coll\II(\TTT)) + \Dist\II(\TT, \TTT).
\end{equation}
\end{lemm}

\begin{proof}
Let $(\TT_0, \Ldots, \TT_\ell)$ be a path from~$\TT$ to~$\TTT$
in~$\Ass{\size\TT}$. By Lemma~\ref{L:CollPair}, the sequence
$(\coll\II(\TT_0), ...,
\coll\II(\TT_\ell))$ is a path from $\coll\II(\TT)$ to~$\coll\II(\TTT)$
in~$\Ass\nnb$, and the number of nontrival pairs in this path is the
number of non-$\II$-collapsing pairs in $(\TT_0, \Ldots, \TT_\ell)$.
Therefore, we have $\ell \ge \dist(\coll\II(\TT),
\coll\II(\TTT)) + \Dist\II(\TT, \TTT)$.
\end{proof}

\begin{rema}
By Lemma~\ref{L:CollPair}, the inequality~\eqref{E:DistColl} is an equality for $\dist(\TT, \TTT) \le 1$.
This need not be true in general. For instance, let $\TT = \Sp{1100}$ and $\TTT = \Sp{0011}$. We saw in
Figure~\ref{F:K4Colours} that the distance between~$\TT$ and~$\TTT$
is~$4$. Now, we have
$\coll{\{4,5\}}(\TT) = \coll{\{4,5\}}(\TTT) = \Sp{11}$, hence $\dist(\coll{\{4,5\}}(\TT), \coll{\{4,5\}}(\TTT)) = 0$.
On the other hand, it can be checked on Figure~\ref{F:K4Colours} that
there exists a path from~$\TT$ to~$\TTT$ in~$\Ass4$ that contains only
two $\{4,5\}$-collapsing pairs, namely
$$\begin{CD}
\TT
@>\text{not}>\{4,5\}-\text{coll.}>
\VR(6,3)
\begin{picture}(9,0)(0,3)
\put(0,0){\includegraphics{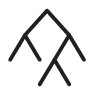}}
\end{picture}
@>>\{4,5\}-\text{coll.}>
\VR(4,3)
\begin{picture}(9,0)(0,3)
\put(0,0){\includegraphics{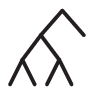}}
\end{picture}
@>>\{4,5\}-\text{coll.}>
\VR(4,3)
\begin{picture}(9,0)(0,3)
\put(0,0){\includegraphics{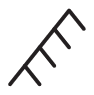}}
\end{picture}
@>\text{not}>\{4,5\}-\text{coll.}>
\TTT\\
@VV\coll{\{4,5\}}V
@VV\coll{\{4,5\}}V
@VV\coll{\{4,5\}}V
@VV\coll{\{4,5\}}V
@V\coll{\{4,5\}}VV\\
\VR(4,3)
\begin{picture}(4,0)(0,1)
\put(0,0){\includegraphics{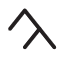}}
\end{picture}
@>>>
\VR(4,3)
\begin{picture}(6,0)(0,1)
\put(0,0){\includegraphics{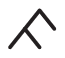}}
\end{picture}
@=
\VR(4,3)
\begin{picture}(6,0)(0,1)
\put(0,0){\includegraphics{Tree5.eps}}
\end{picture}
@=
\VR(4,3)
\begin{picture}(6,0)(0,1)
\put(0,0){\includegraphics{Tree5.eps}}
\end{picture}
@>>>
\VR(4,3)
\begin{picture}(6,0)(0,1)
\put(0,0){\includegraphics{Tree4.eps}}
\end{picture}
\end{CD}
$$
so we have $\Dist{\{4,5\}}(\TT, \TTT) \le 2$ (actually $=2$).
\end{rema}

\subsection{Double collapsing}

Technically, it will be convenient to use two collapsings at a time, with
respect to sets that are strongly disjoint in the following sense.

\begin{lemm}
\label{L:BiColl}
Assume that $\II, \JJ$ satisfy the condition
\begin{equation}
\label{E:BiCollCond}
\forall\ii \in \II \ \forall \jj \in \JJ\ (\ [\ii, \jj] \not\ince \II \cup \JJ\  ).
\end{equation}
Then, for all trees~$\TT, \TTT$, we have
\begin{equation}
\label{E:BiColl}
\dist(\TT, \TTT) \ge \dist(\coll\II(\TT), \coll\II(\TTT) + \Dist\II(\coll\JJ(\TT), \coll\JJ(\TTT)).
\end{equation}
\end{lemm}

\begin{proof}
We claim that the inequality
\begin{equation}
\label{E:BiCollBase}
\Dist\II(\TT, \TTT) \ge \Dist\II(\coll\JJ(\TT), \coll\JJ(\TTT))
\end{equation}
holds for all trees~$\TT, \TTT$. Indeed, let $(\TT_0, \Ldots, \TT_\ell)$ be
a path from~$\TT$ to~$\TTT$. Then
$(\coll\JJ(\TT_0), \Ldots,
\coll\JJ(\TT_\ell))$ is a (possibly redundant) path from~$\coll\JJ(\TT)$
to~$\coll\JJ(\TT_\ell)$.  Put $\ell = \Dist\II(\coll\JJ(\TT),
\coll\JJ(\TTT))$. By definition of~$\Dist\II$, there must be at least
$\ell$~pairs
$(\TT_\rr, \TT_\rrp)$ satisfying
\begin{equation}
\label{L:BiColl2}
\coll\JJ(\TT_\rr)) \not= \coll\JJ(\TT_\rrp)
\text{\quad and\quad }
\coll\II(\coll\JJ(\TT_\rr)) = \coll\II (\coll\JJ(\TT_\rrp)).
\end{equation}
As $\coll\II(\coll\JJ(-)) = \coll{\II\cup\JJ}(-)$ always holds,
\eqref{L:BiColl2} means that $(\TT_\rr, \TT_\rrp)$ is not
$\JJ$-collapsing, and is $(\II \cup \JJ)$-collapsing. Now
Condition~\eqref{L:BiColl} implies that every interval that is included
in~$\II \cup \JJ$ is included in~$\II$, or is included in~$\JJ$. Owing to
the criterion of Lemma~\ref{L:CollPair}, we deduce that $(\TT_\rr,
\TT_\rrp)$ is
$\II$-collapsing, and, therefore, we have $\Dist\II(\TT, \TTT) \ge \ell$.

Then, using~\eqref{L:DistColl}, we obtain
\begin{align*}
\dist(\TT, \TTT) 
&\ge \dist(\coll\II(\TT), \coll\II(\TTT)) + \Dist\II(\TT, \TTT)\\ 
&\ge \dist(\coll\II(\TT), \coll\II(\TTT)) + \Dist\II(\coll\JJ(\TT), \coll\JJ(\TTT)),
\end{align*}
which is the expected inequality~\eqref{E:BiColl}.
\end{proof}

\subsection{Application: a lower bound in $5\nn/3 + O(1)$}
\label{S:5/3}

Lemma~\ref{L:DistColl} provides a natural method for establishing a lower bound on the distance~$\dist(\TT,
\TTT)$ in an inductive way: if $\dist(\TTb, \TTTb) \ge \ellb$ is known,
and we can find a set~$\II$ satisfying $\coll\II(\TT) = \TTb$ and
$\coll\II(\TTT) = \TTTb$, then it suffices to show
that the minimal number of
$\II$-collapsing steps from~$\TT$ to~$\TTT$ is at least~$\kk$ to deduce $\dist(\TT, \TTT) \ge \ellb + \kk$.
We shall now apply this principle to deduce from the results of Section~\ref{S:Bicomb}, which provide a family with
distance $3\nn/2 + O(1)$, a new family achieving distance $5\nn/3 + O(1)$. 

\begin{prop}
\label{P:Tricomb}
For $\TT = \Sp{1^\pp0^\pp1^\pp}$, $\TTT = \Sp{0^\pp(10)^\pp}$ with~$\pp \ge 1$, we have
\begin{equation}
\dist(\TT, \TTT) \ge 5\pp- 4.
\end{equation}
\end{prop}

As $\TT$ and $\TTT$ above have size~$3\pp$, we deduce

\begin{coro}
For $\nn = 3 \pmod 3$, we have $\diam(\nn) \ge \frac53\nn - 4$.
\end{coro}

\begin{figure}[htb]
\begin{picture}(116,62)(0,3)
\put(1,3.5){\includegraphics{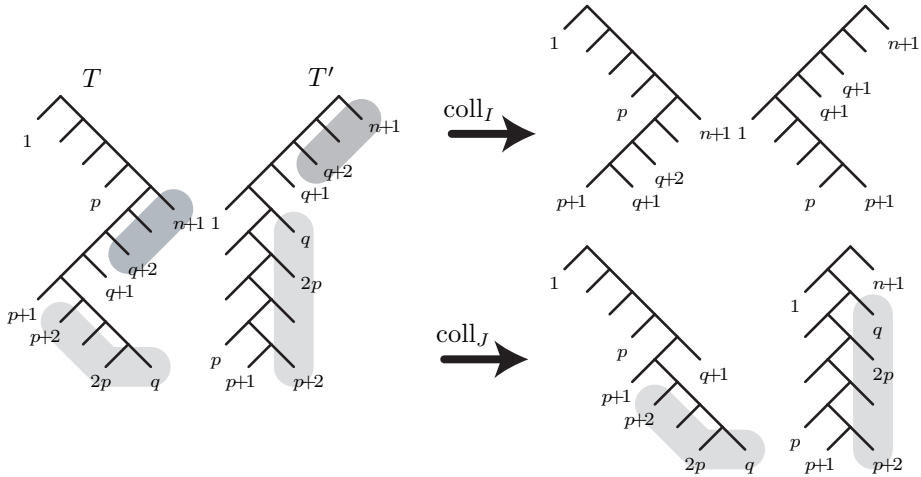}}
\put(8,55){$\TT$}
\put(0,47){$\scriptstyle 1$}
\put(9,39){$\scriptstyle \pp$}
\put(-2,24){$\scriptstyle \pp\!+\!1$}
\put(1,21){$\scriptstyle \pp\!+\!2$}
\put(9,15){$\scriptstyle 2\pp$}
\put(17,15){$\scriptstyle \qq$}
\put(11,27){$\scriptstyle \qq\plus1$}
\put(14,30){$\scriptstyle \qq\plus2$}
\put(20,36){$\scriptstyle \nn\plus1$}

\put(38,55){$\TTT$}
\put(25,36){$\scriptstyle 1$}
\put(25,18){$\scriptstyle \pp$}
\put(27,15){$\scriptstyle \pp\plus1$}
\put(36,15){$\scriptstyle \pp\plus2$}
\put(37,28){$\scriptstyle 2\pp$}
\put(37,34){$\scriptstyle \qq$}
\put(37,40){$\scriptstyle \qq\plus1$}
\put(40,43){$\scriptstyle \qq\plus2$}
\put(46,49){$\scriptstyle \nn\plus1$}

\put(56,51){$\coll\II$}
\put(70,60){$\scriptstyle 1$}
\put(79,51){$\scriptstyle \pp$}
\put(71,39){$\scriptstyle \pp\plus1$}
\put(81,39){$\scriptstyle \qq\plus1$}
\put(84,42){$\scriptstyle \qq\plus2$}
\put(90,48){$\scriptstyle \nn\plus1$}

\put(95,48){$\scriptstyle 1$}
\put(104,39){$\scriptstyle \pp$}
\put(112,39){$\scriptstyle \pp\plus1$}
\put(106,51){$\scriptstyle \qq\plus1$}
\put(109,54){$\scriptstyle \qq\plus1$}
\put(115,60){$\scriptstyle \nn\plus1$}

\put(55,21){$\coll\JJ$}
\put(70,28){$\scriptstyle 1$}
\put(79,19){$\scriptstyle \pp$}
\put(77,13){$\scriptstyle \pp\plus1$}
\put(80,10){$\scriptstyle \pp\plus2$}
\put(88,4){$\scriptstyle2\pp$}
\put(96,4){$\scriptstyle \qq$}
\put(90,16){$\scriptstyle \qq\plus1$}

\put(102,25){$\scriptstyle 1$}
\put(102,7){$\scriptstyle \pp$}
\put(104,4){$\scriptstyle \pp\plus1$}
\put(113,4){$\scriptstyle \pp\plus2$}
\put(113,16){$\scriptstyle 2\pp$}
\put(113,22){$\scriptstyle \qq$}
\put(113,28){$\scriptstyle \nn\plus1$}

\end{picture}
\caption{\sf\smaller The trees of Proposition~\ref{P:Tricomb}---here with
$\pp =\nobreak 4$---and the proof of the latter from
Lemma~\ref{L:Tricomb}: collapsing
$\II$ (light grey labels) leads to $2$-combs, whose distance is known; collapsing~$\JJ$
(dark grey labels) leads to trees that are, up to shifting the labels, those of
Lemma~\ref{L:Tricomb}.}
\label{F:Tricomb}
\end{figure}

To prove Proposition~\ref{P:Tricomb}, we shall use 
Proposition~\ref{P:Bicomb} (or~\ref{P:BicombBis}) and a convenient
collapsing. Fix some~$\pp$, and let $\TTb$ and $\TTTb$ be the
coresponding trees of Proposition~\ref{P:Bicomb}, namely
$\Sp{1^\pp0^\pp}$ and $\Sp{0^\pp1^\pp}$. Collapsing~$\TT$
into~$\TTb$ is easy: $\TT$ is a zigzag
of three alternating length~$\pp$ combs (``tricomb''), whereas $\TTb$ is
a zigzag of two alternating combs (``bicomb''), so that we can
project~$\TT$ to~$\TTb$ by collapsing all labels from~$\pp+2$ to $2\pp
+ 1$. It then turns out that collapsing the same labels in~$\TTT$ leads
to~$\TTTb$. By Proposition~\ref{P:Bicomb}, the distance of~$\TTb$
and~$\TTTb$ is~$3\pp - 2$, so, owing to Lemma~\ref{L:DistColl}, in order
to establish Proposition~\ref{P:Tricomb}, it is enough to prove
\begin{equation}
\label{E:Tricomb}
\Dist{[\pp+2, 2\pp+1]}(\TT, \TTT) \ge 2\pp-2,
\end{equation}
\ie, to prove that each sequence of rotations from~$\TT$ to~$\TTT$ contains at least $\pp-2$ pairs that are $[\pp+2,
2\pp+1]$-collapsing.

Instead of working with the trees of Proposition~\ref{P:Tricomb} themselves, it will be more convenient to use a
second, auxiliary collapsing, and to use Lemma~\ref{L:BiColl}. We shall prove:

\begin{lemm}
\label{L:Tricomb}
For $\TTb = \Sp{1^\pp01^\pp}$, $\TTTb = \Sp{(01)^\pp0}$ with~$\pp \ge 1$, we have
\begin{equation}
\label{E:Tricomb2}
\Dist{[\pp+2, 2\pp+1]}(\TTb, \TTTb) \ge 2\pp-2.
\end{equation}
\end{lemm}

\begin{proof}[Proof of Proposition~\ref{P:Tricomb} from Lemma~\ref{L:Tricomb}]
Let $\II = [\pp+2, 2\pp+1]$ and
\linebreak
 $\JJ = [2\pp+3, 3\pp+1]$. Then the sets~$\II$ and~$\JJ$ satisfy
Condition~\eqref{E:BiColl}, and we have $\TTb = \coll\JJ(\TT)$ and 
$\TTTb = \coll\JJ(\TTT)$. Lemma~\ref{L:BiColl} then gives
$$\dist(\TT, \TTT)  \ge \dist(\coll\II(\TT), \coll\II(\TTT)) + \Dist\II(\TTb, \TTTb).$$
As can be checked on
Figure~\ref{F:Tricomb}, we have $\coll\II(\TT) = \Sp{1^\pp0^\pp}$ and
$\Sp{0^\pp1^\pp}$. Using Proposition~\ref{P:Bicomb} and \eqref{E:Tricomb2}, we deduce
$$\dist(\TT, \TTT)  \ge (3\pp - 2) + (2\pp - 2) = 5\pp - 4,$$
as expected.
\end{proof}

\begin{proof}[Proof of Lemma~\ref{L:Tricomb}]
The argument is similar to the one used for Proposition~\ref{P:BicombBis}, with the
additional difficulty that we need pairs that are $[\pp+2, 2\pp + 1]$-collapsing. Put $\nnb =
2\pp + 1$, $\II = [\pp+2, 2\pp]$, and, for $\pp+2 \le \aa  \le 2\pp$, say that a base pair is
\begin{tabbing}
\quad
\=-\emph{special of type $\tI_\aa$} \hspace{0.5em}
\=if it is $\ \Qp(..., \aa, \aa, ...)$ \\
\>-\emph{special of type $\tIIp_\aa$} 
\>if it is $\ \Qp(..., \lse\ppp, \aa, \lse\nnb)$\\
\>-\emph{special of type $\tIIm_\aa$} 
\>if it is $\ \Qm(\lse\pp, ..., \gse\ppp, \aa)$\\
\>-\emph{special of type $\tIII_\aa$} 
\>if it is $\ \Qm(..., \aap, ..., \nnb+1)$.
\end{tabbing}
It is straightforward that a special pair has a unique type, and that all special pairs are $\II$-collapsing:
with obvious notation, we have $[\nu_2, \nu_3] \ince \II$ for
type~$\tI_\aa$, and $[\nu_3+1, \nu_4] \ince
\II$ for types~$\tIIp_\aa$, $\tIIm_\aa$, and~$\tIII_\aa$, which is enough to conclude using the criterion of
Lemma~\ref{L:CollPair}.

Let $(\TT_0, \Ldots, \TT_\ell)$ be a path from~$\TTb$ to~$\TTTb$
in~$\Ass\nnb$. Choose $\aa$ in~$[\pp+2, 2\pp]$. First, as can be read on
Figure~\ref{F:Tricomb}, we have 
$$\aam \not\cov\TTb \aa
\text{\quad and \quad}
\aam \cov\TTTb \aa.$$
Lemma~\ref{L:KeyLemma1} guarantees that $(\TT_0, \Ldots, \TT_\ell)$
contains a pair $\Qp(\lse\aam,
\gs\aam, \aa, ...)$, \ie, a pair of type~$\tI_\aa$.

Next, let $\aa' = \nnb + 1 - \aa$. Then we have $2 \le \aa' \le \pp$, and we read the relations
$$\aa' \not\cov\TTb \aa, \quad
\aa' \cov\TTb \aa,
\text{\quad and \quad}
\aa'+1 \ccov\TTTb \aa.$$
Applying Lemma~\ref{L:PPPrinciple} with $\ii = \aa'$, $\jj = \aa$, and $\kk = \ppp$
guarantees the existence of~$\rr$ such that
$(\TT_\rr, \TT_\rrp)$ is
$$\Qp(\lse\aa', \aa'\ls...\lse\ppp, \aa,
...)
\mbox{\ or\ }
\Qm(\lse\aa',
\aa'\ls...\lse\ppp, \gse\ppp, \aa).$$In the latter case, we have a special
pair of type~$\tIIm_\aa$. In the former case, we have a special pair of
type~$\tIIp_\aa$ provided the last parameter, namely
$\Name4{\TT_\rr}{\TT_\rrp}$, is at most~$\nnb$. Now assume this is not
the case, \ie, we have $\Name4{\TT_\rr}{\TT_\rrp} = \nnb+1$. By
Lemma~\ref{L:NameCov}$(iv)$, we have $\aap \ccove{\TT_\rr} \nnb+1$,
hence $\aap \ccov{\TT_\rr} \nnb+1$ as $\aap \le \nnb$ holds by
hypothesis. But, by hypothesis, we have $\aap \not\ccov\TT \nnb+1$. By
Lemma~\ref{L:KeyLemma2}, there must exist~$\ss \le \rr$ such that
$(\TT_\ss, \TT_\ssp)$ is $\Qm(..., \aap, ...,
\nnb+1)$, hence special of type~$\tIII_\aa$.

Hence, for  each of the $\ppm$ values $\pp+2, \Ldots, 2\pp$, the path $(\TT_0, \Ldots, \TT_\ell)$ contains a pair
of type~$\tI_\aa$, and a pair of type~$\tIIp_\aa, \tIIm_\aa$, or~$\tIII_\aa$. Hence $(\TT_0, \Ldots, \TT_\ell)$
contains
\begin{tabbing}
\hspace{3em}
\=$\ppm$
\=pairs of type~$\tI$,\\
\>$\ppm$ 
\>pairs of type~$\tIIpm$ or~$\tIII$,
\end{tabbing}
hence at least $2\pp-2$ pairs that are $\II$-collapsing.
\end{proof}

\begin{rema}
The previous result is optimal. It is not difficult to construct an explicit path
of length~$5\pp-4$ from~$\TT$ to~$\TTT$, and to check that its
projection is a path of length~$2\pp - 2$ from~$\TTb$ to~$\TTTb$, all of
which steps are $\II$-collapsing. Also, one can observe that, in the above
proof, the final argument showing the existence of a special pair of
type~$\tIII_\aa$ also shows the existence of a pair $\Qp(..., \aap, ...,
\nnb+1)$. It follows that each path visiting a vertex satisfying $\aa \ccov{}
\nnb+1$ for some~$\aa \ge \pp+3$ contains at least $2\pp -1$ pairs that
are $\II$-collapsing. Hence such a path has length at least~$5\pp - 3$ and
it is not geodesic: moving $\nnb+1$ to the right is never optimal.
\end{rema}

\subsection{Alternative families}
\label{S:Another}

To conclude this section, we mention still another family witnessing a
growth rate of the form $3\nn/2 + O(1)$. The analysis of this
family, which heavily uses collapsing, is easier than that of
Proposition~\ref{P:BicombBis}, but, contrary to the arguments developed
above, it does  not seem to extend for proving stronger results. The
trees we consider are zigzags, with a slight change near the ends---though
seemingly minor, that change modifies distances completely.

\begin{prop}
\label{P:Zigzag}
For $\mm \ge 0$, we have
\begin{equation}
\label{E:Zigzag}
\dist(\Sp{1(10)^\mm0}, \Sp{0(01)^\mm1}) = 3\mm + 1.
\end{equation}
\end{prop}

\begin{figure}[htb]
\begin{picture}(114,51)(0,0)
\put(-0.5,0){\includegraphics{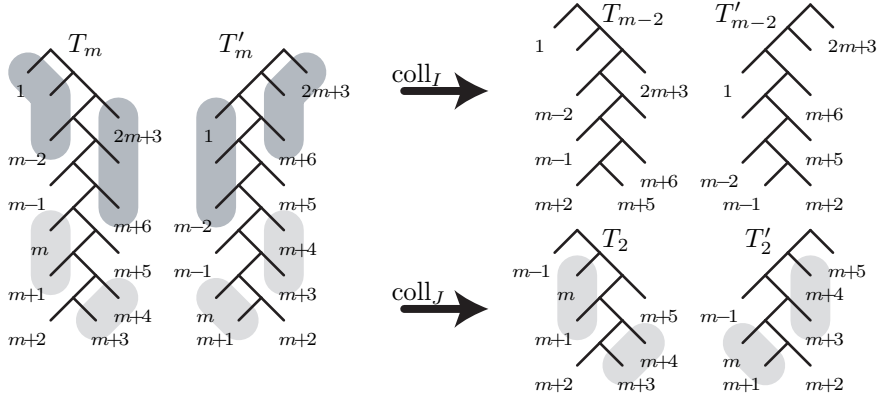}}
\put(8,45){$\TT_\mm$}
\put(1,39){$\scriptstyle 1$}
\put(0,30){$\scriptstyle \mm\minus2$}
\put(0,24){$\scriptstyle \mm\minus1$}
\put(3,18){$\scriptstyle \mm$}
\put(0,12){$\scriptstyle \mm\plus1$}
\put(0,6){$\scriptstyle \mm\plus2$}
\put(11,6){$\scriptstyle \mm\plus3$}
\put(14,9){$\scriptstyle \mm\plus4$}
\put(14,15){$\scriptstyle \mm\plus5$}
\put(14,21){$\scriptstyle \mm\plus6$}
\put(14,33){$\scriptstyle 2\mm\plus3$}

\put(28,45){$\TTT_\mm$}
\put(26,33){$\scriptstyle 1$}
\put(22,21){$\scriptstyle \mm\minus2$}
\put(22,15){$\scriptstyle \mm\minus1$}
\put(25,9){$\scriptstyle \mm$}
\put(25,6){$\scriptstyle \mm\plus1$}
\put(36,6){$\scriptstyle \mm\plus2$}
\put(36,12){$\scriptstyle \mm\plus3$}
\put(36,18){$\scriptstyle \mm\plus4$}
\put(36,24){$\scriptstyle \mm\plus5$}
\put(36,30){$\scriptstyle \mm\plus6$}
\put(39,39){$\scriptstyle 2\mm\plus3$}

\put(51,41){$\coll\II$}
\put(79,49){$\TT_{\mm-2}$}
\put(70,45){$\scriptstyle 1$}
\put(70,36){$\scriptstyle \mm\minus2$}
\put(70,30){$\scriptstyle \mm\minus1$}
\put(70,24){$\scriptstyle \mm\plus2$}
\put(81,24){$\scriptstyle \mm\plus5$}
\put(84,27){$\scriptstyle \mm\plus6$}
\put(84,39){$\scriptstyle 2\mm\plus3$}

\put(94,49){$\TTT_{\mm-2}$}
\put(95,39){$\scriptstyle 1$}
\put(92,27){$\scriptstyle \mm\minus2$}
\put(95,24){$\scriptstyle \mm\minus1$}
\put(106,24){$\scriptstyle \mm\plus2$}
\put(106,30){$\scriptstyle \mm\plus5$}
\put(106,36){$\scriptstyle \mm\plus6$}
\put(109,45){$\scriptstyle 2\mm\plus3$}

\put(51,12){$\coll\JJ$}
\put(79,19){$\TT_2$}
\put(67,15){$\scriptstyle \mm\minus1$}
\put(73,12){$\scriptstyle \mm$}
\put(70,6){$\scriptstyle \mm\plus1$}
\put(70,0){$\scriptstyle \mm\plus2$}
\put(81,0){$\scriptstyle \mm\plus3$}
\put(84,3){$\scriptstyle \mm\plus4$}
\put(84,9){$\scriptstyle \mm\plus5$}

\put(98,19){$\TTT_2$}
\put(92,9){$\scriptstyle \mm\minus1$}
\put(95,3){$\scriptstyle \mm$}
\put(95,0){$\scriptstyle \mm\plus1$}
\put(106,0){$\scriptstyle \mm\plus2$}
\put(106,6){$\scriptstyle \mm\plus3$}
\put(106,12){$\scriptstyle \mm\plus4$}
\put(109,15){$\scriptstyle \mm\plus5$}
\end{picture}
\caption{\sf\smaller The trees of Proposition~\ref{P:Zigzag}---here with
$\mm = 5$: collapsing
$\II$ (light grey labels) from~$\TT_\mm$ and $\TTT_\mm$ leads (up to a relabeling) to~$\TT_{\mm-2}$
and~$\TTT_{\mm-2}$, while collapsing~$\JJ$ (dark grey labels) leads (up to a relabeling) to~$\TT_2$
and~$\TTT_2$.}
\label{F:Zigzag}
\end{figure}

\begin{proof}
Let $\TT_\mm = \Sp{1(10)^\mm0}$ and $\TTT_\mm = \Sp{0(01)^\mm1}$. We use induction on~$\mm$. For
$\mm \le 1$, the result is easily checked by a direct computation. Assume $\mm \ge 2$. Put
$$\II_\mm = \{\mm, \mmp, \mm+3, \mm+4\}
\quad\text{and} \quad
\JJ_\mm = [1, \mm-2] \cup [\mm + 6, 2\mm+3].$$
As can be read on Figure~\ref{F:Zigzag}, we have 
$$\coll{\II_\mm}(\TT_\mm) =  \TT_{\mm-2}
\quad\text{and} \quad
\coll{\II_\mm}(\TTT_\mm) =  \TTT_{\mm-2},$$
as well as
$$\coll{\JJ_\mm}(\TT_\mm) =  \TT_2
\quad\text{and} \quad
\coll{\JJ_\mm}(\TTT_\mm) =  \TTT_2.$$
The sets~$\II_\mm$ and~$\JJ_\mm$ satisfy the disjointness condition~\eqref{E:BiCollCond}. Applying
Lemma~\ref{L:DistColl}, we deduce
$$\dist(\TT_\mm, \TTT_\mm) \ge
\dist(\TT_{\mm-2}, \TTT_{\mm-2}) + \Dist{\II_2}(\TT_2, \TTT_2).$$
A brute force verification---or a proof using the techniques of
Section~\ref{S:Principle}---gives
$$\Dist{\II_2}(\TT_2, \TTT_2) = 6,$$
and $\dist(\TT_\mm, \TTT_\mm) \ge 3\mm + 1$ follows inductively. The other inequality, whence
\eqref{E:Zigzag}, are easily checked by a direct computation.
\end{proof}

The remarkably simple proof of Proposition~\ref{P:Zigzag} relies on the conjunction of two properties. 
First, the two families of trees we consider are stable under two types of
collapsing simultaneously: whether we collapse from the top or from the
bottom, we can manage to remain in the same family. Second, the
involved collapsing are perfect, in the sense that the
inequality~\eqref{E:DistColl} turns out to be an equality, a necessary
condition if we are to obtain an exact value.

It is not difficult to obtain other families satisfying one of the above two properties. In particular, for each~$\pp \ge
1$, the thin trees $\Sp{1^\pp (10)^\mm 0^\pp}$ and $\Sp{0^\pp (01)^\mm 1^\pp}$ are eligible.
Proposition~\ref{P:Zigzag} then corresponds to $\pp = 1$. The choice $\pp = 0$ is uninteresting, but $\pp \ge 2$
gives seemingly large distances. Applying the scheme above leads to looking for the unique parameter
$$\Dist{([\pp+1, 3\pp+3]\setminus\{2\pp+2\})}(\Sp{1^\pp(10)^{\ppp}0^\pp}, \Sp{0^\pp(01)^{\ppp}1^\pp}),$$
which we have seen is~$6$ for $\pp = 1$. For $\pp = 2$, one obtains the value~$10$, but it is then easy to see that
such a value cannot give an equality in~\eqref{E:DistColl}. In this way, one obtains lower bounds for the distances of
the involved trees, but these bounds are not sharp, and it seems hard to
obtain very strong results. For instance, for $\pp = 2$, collapsing six leaves
guarantees ten collapsing steps, and one cannot obtain more than $5\nn/3
+ O(1)$.

\section{A lower bound in $2\nn + O(\sqrt{\nn})$}
\label{S:Sqrt}

It is not hard to repeat the argument of Section~\ref{S:5/3} so as to
construct explicit trees of size~$\nn$ at distance $7\nn/4 + O(1)$.
Unfortunately, a furher iteration seems difficult and, in order to go
farther, we shall have to develop a more intricate argument---yet the
principle always remains the same. The main difference is that, now, we
will collapse a size~$2\pp$ bicomb rather than a size~$\pp$ comb.

\subsection{Starting from an $\mm$-comb}

The main result of this section is analogous to Proposition~\ref{P:Tricomb}, but the source tree is a
$2\mm$-comb rather than a tricomb, \ie, it is a tree obtained by stacking
$2\mm$ left and right combs, alternately. The target tree is again a zigzag
preceded by a short left comb---see Figure~\ref{F:Multicomb}.

\begin{prop}
\label{P:Multicomb}
For~$\mm, \pp \ge 1$, we have 
\begin{equation}
\label{E:Multicomb}
\dist(\Sp{(1^\pp0^\pp)^\mm}, \Sp{0^\pp(10)^{(\mmm)\pp}1^\pp})
\ge 4\mm\pp - 3\mm -\pp + 1.
\end{equation}
\end{prop}

\begin{figure}[htb]
\begin{picture}(114,80)(0,0)
\put(0,-1.5){\includegraphics{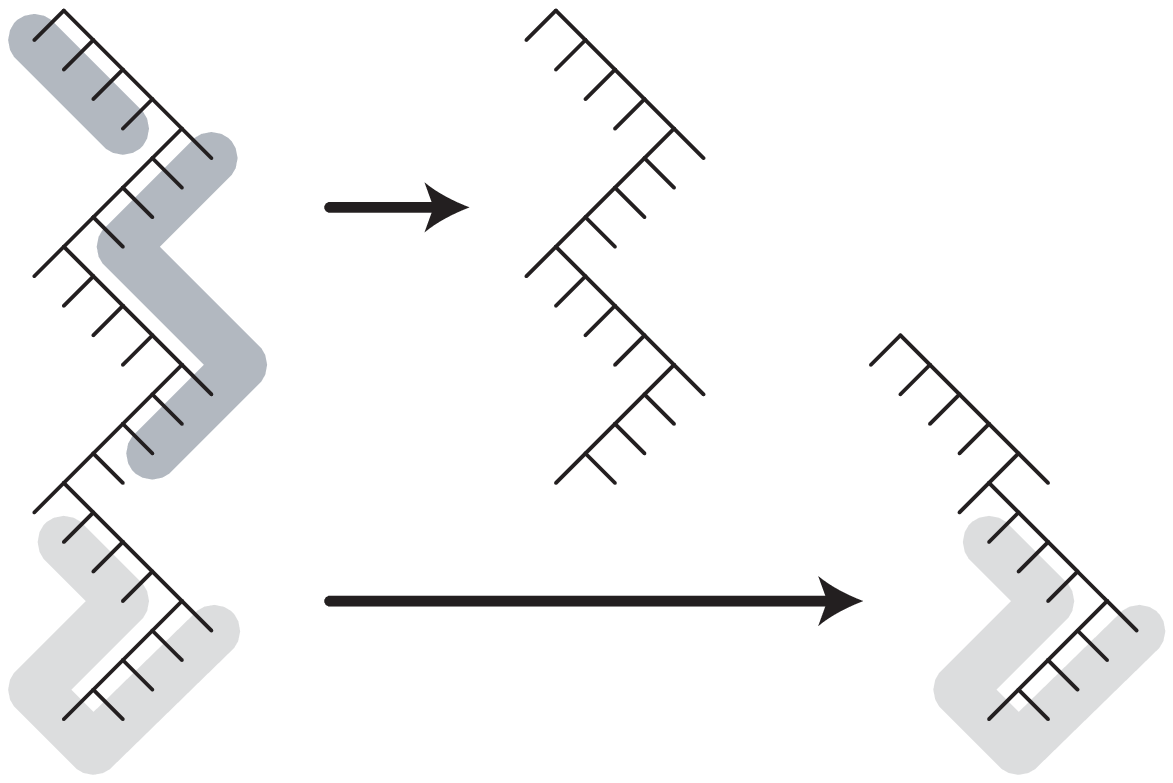}}
\put(10,75){$\TT_{\mm,\pp}$}
\put(1,71){$\scriptstyle 1$}
\put(4,62){$\scriptstyle (\mm\minus2)\pp$}
\put(-7,47){$\scriptstyle (\mm\minus2)\pp\plus1$}
\put(3,38){$\scriptstyle (\mm\minus1)\pp$}
\put(-7,23){$\scriptstyle (\mm\minus1)\pp\plus1$}
\put(-4,20){$\scriptstyle (\mm\minus1)\pp\plus2$}
\put(8,14){$\scriptstyle \mm\pp$}
\put(0,2){$\scriptstyle \mm\pp\plus1$}
\put(10,2){$\scriptstyle \mm\pp\plus2$}
\put(19,11){$\scriptstyle (\mm\plus1)\pp\plus1$}
\put(10,26){$\scriptstyle (\mm\plus1)\pp\plus2$}
\put(13,29){$\scriptstyle (\mm\plus1)\pp\plus3$}
\put(19,35){$\scriptstyle (\mm\plus2)\pp\plus1$}
\put(10,50){$\scriptstyle (2\mm\minus1)\pp\plus2$}
\put(19,59){$\scriptstyle 2\mm\pp\plus1$}

\put(35,59){$\coll\II$}
\put(60,75){$\TT_{\mmm,\pp}$}
\put(51,71){$\scriptstyle 1$}
\put(54,62){$\scriptstyle (\mm\minus2)\pp$}
\put(43,47){$\scriptstyle (\mm\minus2)\pp\plus1$}
\put(53,38){$\scriptstyle (\mm\minus1)\pp$}
\put(46,26){$\scriptstyle (\mm\minus1)\pp\plus1$}
\put(60,26){$\scriptstyle (\mm\plus1)\pp\plus2$}
\put(63,29){$\scriptstyle (\mm\plus1)\pp\plus3$}
\put(69,35){$\scriptstyle (\mm\plus2)\pp\plus1$}
\put(60,50){$\scriptstyle (2\mm\minus1)\pp\plus2$}
\put(69,59){$\scriptstyle 2\mm\pp\plus1$}

\put(35,19){$\coll\JJ$}
\put(95,44){$\TTb_\pp$}
\put(78,38){$\scriptstyle (\mm\minus2)\pp\plus1$}
\put(88,29){$\scriptstyle (\mm\minus1)\pp$}
\put(88,23){$\scriptstyle (\mm\minus1)\pp\plus1$}
\put(91,20){$\scriptstyle (\mm\minus1)\pp\plus2$}
\put(103,14){$\scriptstyle \mm\pp$}
\put(95,2){$\scriptstyle \mm\pp\plus1$}
\put(105,2){$\scriptstyle \mm\pp\plus2$}
\put(114,11){$\scriptstyle (\mm\plus1)\pp\plus1$}
\put(105,26){$\scriptstyle (\mm\plus1)\pp\plus2$}
\end{picture}

\begin{picture}(114,84)(0,0)
\put(0,-1.5){\includegraphics{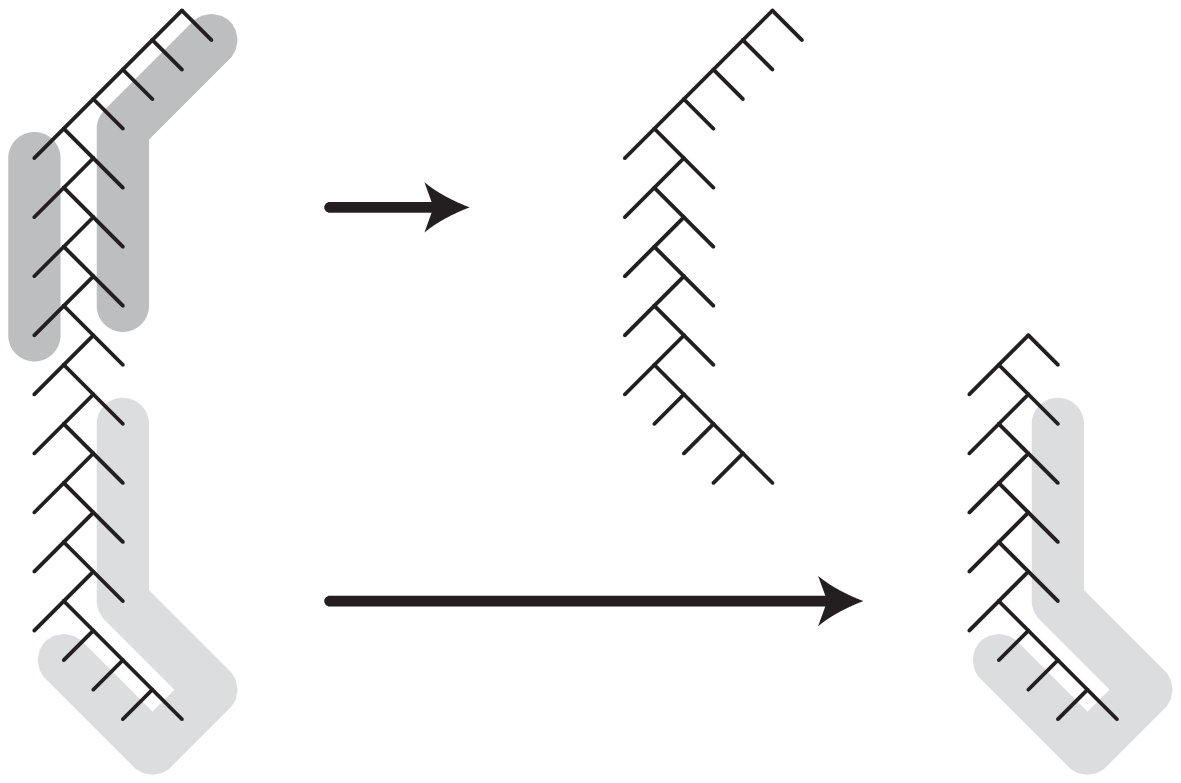}}
\put(5,75){$\TTT_{\mm,\pp}$}
\put(1,60){$\scriptstyle 1$}
\put(-4,41){$\scriptstyle (\mm\minus2)\pp$}
\put(-7,35){$\scriptstyle (\mm\minus2)\pp\plus1$}
\put(-5,17){$\scriptstyle (\mm\minus1)\pp$}
\put(-7,11){$\scriptstyle (\mm\minus1)\pp\plus1$}
\put(-4,8){$\scriptstyle (\mm\minus1)\pp\plus2$}
\put(8,2){$\scriptstyle \mm\pp$}
\put(17,2){$\scriptstyle \mm\pp\plus1$}
\put(11,14){$\scriptstyle \mm\pp\plus2$}
\put(11,32){$\scriptstyle (\mm\plus1)\pp\plus1$}
\put(11,38){$\scriptstyle (\mm\plus1)\pp\plus2$}
\put(11,44){$\scriptstyle (\mm\plus1)\pp\plus3$}
\put(11,56){$\scriptstyle (\mm\plus2)\pp\plus1$}
\put(11,62){$\scriptstyle (2\mm\minus1)\pp\plus2$}
\put(20,71){$\scriptstyle 2\mm\pp\plus1$}

\put(35,59){$\coll\II$}
\put(57,75){$\coll\II(\TTT_{\mm,\pp})$}
\put(5,75){$\TTT_{\mm,\pp}$}
\put(61,60){$\scriptstyle 1$}
\put(56,41){$\scriptstyle (\mm\minus2)\pp$}
\put(53,35){$\scriptstyle (\mm\minus2)\pp\plus1$}
\put(56,32){$\scriptstyle (\mm\minus2)\pp\plus2$}
\put(64,26){$\scriptstyle (\mm\minus1)\pp$}
\put(75,26){$\scriptstyle (\mm\minus1)\pp\plus1$}
\put(71,38){$\scriptstyle (\mm\plus1)\pp\plus2$}
\put(71,44){$\scriptstyle (\mm\plus1)\pp\plus3$}
\put(71,56){$\scriptstyle (\mm\plus2)\pp\plus1$}
\put(71,62){$\scriptstyle (2\mm\minus1)\pp\plus2$}
\put(80,71){$\scriptstyle 2\mm\pp\plus1$}

\put(35,19){$\coll\JJ$}
\put(98,45){$\TTTb_{\pp}$}
\put(87,35){$\scriptstyle (\mm\minus2)\pp+1$}
\put(89,29){$\scriptstyle (\mm\minus1)\pp$}
\put(88,11){$\scriptstyle (\mm\minus1)\pp\plus1$}
\put(91,8){$\scriptstyle (\mm\minus1)\pp\plus2$}
\put(104,2){$\scriptstyle \mm\pp$}
\put(112,2){$\scriptstyle \mm\pp\plus1$}
\put(106,14){$\scriptstyle \mm\pp\plus2$}
\put(106,32){$\scriptstyle (\mm\plus1)\pp\plus1$}
\put(106,38){$\scriptstyle (\mm\plus1)\pp\plus2$}
\end{picture}

\caption{\sf\smaller The trees of  Proposition~\ref{P:Multicomb}---here
with $\mm = 3$ and $\pp =\nobreak 4$---and the proof of the latter from
Lemma~\ref{L:Multicomb}: collapsing
$\II$ (light grey labels) leads from~$(\TT_{\mm,\pp}, \TTT_{\mm,\pp})$
to~$(\TT_{\mmm,\pp}, \TTT_{\mmm,\pp})$; collapsing~$\JJ$ (dark grey labels) leads
from~$(\TT_{\mm,\pp}, \TTT_{\mm,\pp})$ to the pair $(\TTb_\pp, \TTTb_\pp)$, which does not depend
on~$\mm$ and is, up to shifting the labels, the pair of
Lemma~\ref{L:Multicomb}.}
\label{F:Multicomb}
\end{figure}

By letting $\mm$ stay fixed and $\pp$ vary, we obtain size~$\nn$ trees whose 
distance grows at least---actually, one can check that
\eqref{E:Multicomb} is an equality---as
$(4\mm-1)\nn/(2\mm) + O(1)$. By letting~$\mm$ and $\pp$ vary simultaneously, we obtain a lower bound
in $2\nn + O(\sqrt\nn)$:

\begin{coro}
\label{C:Sqrt}
For $\nn$ of the form 2$\mm^2$, we have 
\begin{align}
\label{E:Sqrt1}
\diam(\nn) 
\ge  2\nn + 1  - 2\, \sqrt{2\nn}.
\end{align}
Moreover, $\diam(\nn) \ge 2\nn - \CC \sqrt{\nn}$ holds for each~$\nn$
with $\CC = \sqrt{70}$.
\end{coro}

\begin{proof}[Proof (of Corollary~\ref{C:Sqrt} from Proposition~\ref{P:Multicomb})]
Assume first $\nn = 2\mm^2$. Then, choosing $\pp = \mm$ and using~\eqref{E:Multicomb}, we find
$$\diam(\nn) = \diam(2\mm\pp) \ge
\dist(\Sp{(1^\pp0^\pp)^\mm}, \Sp{0^\pp(10)^{(\mmm)\pp}1^\pp}) = 4\mm^2 - 4\mm +1,$$
whence~\eqref{E:Sqrt1}. For the general case, choose $\mm = \lfloor \sqrt{5\nn/14} \rfloor$, $\pp = \lfloor
\sqrt{7\nn/10} \rfloor$. Then we have $2\mm\pp \le \nn$, whence
$$\diam(\nn) \ge \diam(2\mm\pp) \ge 
\dist(\Sp{(1^\pp0^\pp)^\mm}, \Sp{0^\pp(10)^{(\mmm)\pp}1^\pp}) =4\mm\pp - 3\mm - \pp +1.$$
by~\eqref{E:Multicomb}. Using
$\mm \le \sqrt{5\nn/14} < \mmp$ and $\pp \le \sqrt{7\nn/10} < \ppp$,
one obtains
$$4 \mm\pp  - 3\mm - \pp +1 
> 2\nn - (7\sqrt{5\nn/14} - 5\sqrt{7\nn/10}) + 5
= 2\nn - \sqrt{70\nn} + 5.
\eqno\square
$$
\let\qed\relax
\end{proof}

The proof of Proposition~\ref{P:Multicomb} uses an induction on the parameter~$\mm$, \ie, on the number of
alternations in the zigzag-tree~$\Sp{(1^\pp0^\pp)^\mm}$. It is not hard to find how to collapse~$\TT_{\mm,
\pp}$ to~$\TT_{\mmm, \pp}$ and $\TTT_{\mm,\pp}$ 
to~$\TTT_{\mmm, \pp}$. Then the key point consists in identifying
sufficiently many collapsing pairs, and this is done in the following result.

\begin{lemm}
\label{L:Multicomb}
Let $\TT = \Sp{1^\pp01^\pp0^\pp}$, $\TTT = \Sp{0(10)^\pp1^\pp}$, and
$\II = [\pppp, 2\pp+1]$. Then, for each $\pp \ge 1$, we have
\begin{equation}
\label{E:Tricomb2}
\Dist\II(\TT, \TTT) \ge 4\pp-3.
\end{equation}
\end{lemm}

\begin{proof}[Proof of Proposition~\ref{P:Multicomb} from Lemma~\ref{L:Multicomb}]
The argument is exactly similar to the one used to deduce Proposition~\ref{P:Tricomb} from
Lemma~\ref{L:Tricomb}. We use induction on~$\mm \ge 1$. The trees $\TT_{1, \pp}$ and $\TTT_{1, \pp}$ are the
$2$-combs $\Sp{1^\pp0^\pp}$ and $\Sp{0^\pp1^\pp}$, and 
Proposition~\ref{P:Bicomb} (or~\ref{P:BicombBis}) gives $\dist(\TT_{1,
\pp}, \TTT_{1, \pp}) \ge 3\pp-2$, in agreement with~\eqref{P:Multicomb}.

Assume now $\mm \ge 2$, and let
$$\II = [(\mmm)\pp+2, (\mmp)\pp+1]
\quad\text{and}\quad
\JJ= [1, (\mm-2)\pp] \cup
[(\mmp)\pp+3, 2\mm\pp+1].$$
Then $\II$ and $\JJ$ satisfy the disjointness condition of Lemma~\ref{L:BiColl}, and we read on
Figure~\ref{F:Multicomb} that $\II$-collapsing maps~$\TT_{\mm,\pp}$ to~$\TT_{\mmm, \pp}$,
and~$\TTT_{\mm,\pp}$ to~$\TTT_{\mmm, \pp}$. On the other hand, $\JJ$-collapsing maps~$\TT_{\mm,\pp}$
to a tree~$\TTb_\pp$ that is, up to shifting the labels, the tree~$\TT$ of Lemma~\ref{L:Multicomb},
and~$\TTT_{\mm,\pp}$ to a tree~$\TTTb_\pp$ that is, up to shifting the labels, the tree~$\TTT$ of
Lemma~\ref{L:Multicomb}. Applying Lemma~\ref{L:BiColl}, the induction hypothesis for~$\mmm$, and
Lemma~\ref{L:Multicomb}, we obtain
\begin{align*}
\dist(&\TT_{\mm,\pp}, \TTT_{\mm,\pp})\\
&\ge \dist(\coll\II(\TT_{\mm,\pp}), \coll\II(\TTT_{\mm,\pp})) +
\Dist\II(\coll\JJ(\TT_{\mm,\pp}),
\coll\JJ(\TTT_{\mm,\pp}))\\
&= \dist(\TT_{\mmm,\pp}, \TTT_{\mmm,\pp}) + \Dist\II(\TTb_\pp, \TTTb_\pp)\\
&\ge (4(\mm\!-\!1)\pp - 3(\mm\!-\!1) - \pp + 1) + (4\pp-3)
= 4\mm\pp - 3\mm - \pp + 1,
\end{align*}
which is \eqref{E:Multicomb}. 
\end{proof}

\subsection{Special pairs}

The proof of Lemma~\ref{L:Multicomb} will occupy the next sections. It is parallel to the proof
of Lemma~\ref{L:Tricomb}, but it is more involved and requires some care,
mainly because we are to introduce many different types of special pairs.

In the sequel, we put $\nn = 3\pp + 1$ and $\qq = 2\pp + 1$. Then $\TT$ and $\TTT$ are
size~$\nn$ trees, displayed in Figure~\ref{F:Key}---compared with the trees~$\TTb_\pp$ and $\TTTb_\pp$ of
Figure~\ref{F:Multicomb}, the only difference is that the labeling has been
standardized. 

\begin{figure}[htb]
\begin{picture}(70,47)(0,1)
\put(1,-0.5){\includegraphics{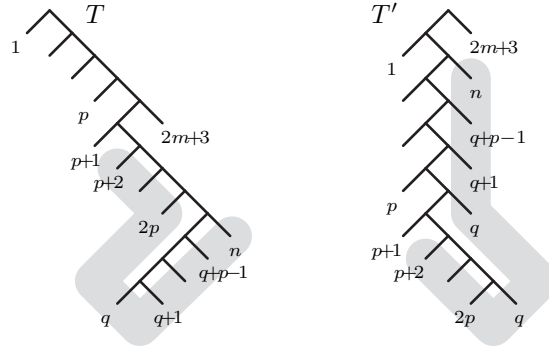}}
\put(10,43){$\TT$}
\put(0,39){$\scriptstyle 1$}
\put(9,30){$\scriptstyle \pp$}
\put(8,24){$\scriptstyle \pp\!+\!1$}
\put(11,21){$\scriptstyle \pp\!+\!2$}
\put(17,15){$\scriptstyle 2\pp$}
\put(12,3){$\scriptstyle \qq$}
\put(19,3){$\scriptstyle \qq\plus1$}
\put(25,9){$\scriptstyle \qq\plus\pp\minus1$}
\put(29,12){$\scriptstyle \nn$}
\put(20,27){$\scriptstyle 2\mm\plus3$}

\put(48,43){$\TTT$}
\put(50,36){$\scriptstyle 1$}
\put(50,18){$\scriptstyle \pp$}
\put(48,12){$\scriptstyle \pp\plus1$}
\put(51,9){$\scriptstyle \pp\plus2$}
\put(59,3){$\scriptstyle 2\pp$}
\put(67,3){$\scriptstyle \qq$}
\put(61,15){$\scriptstyle \qq$}
\put(61,21){$\scriptstyle \qq\plus1$}
\put(61,27){$\scriptstyle \qq\plus\pp-1$}
\put(61,33){$\scriptstyle \nn$}
\put(61,39){$\scriptstyle 2\mm\plus3$}
\end{picture}
\caption{\sf\smaller The trees of Lemma~\ref{L:Multicomb}, here with
$\pp = 4$, and, in grey, the labels of~$\II$, which are those for which we
need to count the $\II$-collapsing pairs.}
\label{F:Key}
\end{figure}

Then we consider the
following eleven {(!)} families of base pairs. We say that a base pair is

\begin{tabbing}
\quad
\=-\emph{special of type $\tIp_\aa$} \hspace{0.5em}
\=if it is $\ \Qp(..., \aa, \qq, ...)$ \hspace{5em}
\=with $\pppp \le \aa \le \qq$,\\
\>-\emph{special of type $\tIm_\aa$} 
\>if it is $\ \Qm(..., \aa, ..., \qq)$ 
\>with $\pppp \le \aa < \qq$,\\
\>-\emph{special of type $\tIIp_\aa$} 
\>if it is $\ \Qp(..., \qq, \aa, ...)$ 
\>with $\qqp \le \aa < \nn$,\\
\>-\emph{special of type $\tIIm_\aa$} 
\>if it is $\ \Qm(\qq, ..., \aa, ...)$ 
\>with $\qqp \le \aa < \nn$,\\
\>-\emph{special of type $\tIIIp_\aa$} 
\>if it is $\ \Qp(..., \aa, \ls\qq, ...)$ 
\>with $\pppp \le \aa < \qq$,\\
\>-\emph{special of type $\tIIIm_\aa$} 
\>if it is $\ \Qm(..., \aa, ..., \nnp)$ 
\>with $\pppp \le \aa \le \qq$,\\
\>-\emph{special of type $\tIVp_\aa$} 
\>if it is $\ \Qp(...,  \gse\pppp \,\&\, {\not=}\qq, \aa, ...)$ 
\>with $\qqp <  \aa \le \nn$,\\
\>-\emph{special of type $\tVp_\aa$} 
\>if it is $\ \Qp(..., \lse\ppp, \aa, \lse\nn)$ 
\>with $\qqp \le \aa < \nn$,\\
\>-\emph{special of type $\tVm_\aa$} 
\>if it is $\ \Qm(..., \lse\ppp, \gs\pp, \aa)$ 
\>with $\qqp \le \aa < \nn$,\\
\>-\emph{special of type $\tVIp_\aa$} 
\>if it is $\ \Qp(..., \aap, ..., \nnp)$ 
\>with $\qqp \le \aa < \nn$,\\
\>-\emph{special of type $\tVIm_\aa$} 
\>if it is $\ \Qm(..., \aap, ..., \nnp)$ 
\>with $\qqp \le \aa < \nn$.\\
\end{tabbing}

\begin{claim}
\label{C:1}
Every special pair is $\II$-collapsing.
\end{claim}

\begin{proof}
As we consider size~$\nn$ trees, every base pair satisfies $\nu_4 \le \nnp$, hence $\nu_3 \le \nn$ and,
thereforen $\nu_3 \in \II$. For all special pairs except those of type~$\tV$, we have $\nu_2 \ge \pppp$,
hence $[\nu_2, \nu_3] \ince \II$. For the special pairs of type~$\tV$, we have $\nu_3 \ge \ppp$ and $\nu_4 \le \nn$,
hence $[\nu_3+1, \nu_4] \ince \II$. In all cases, the criterion of Lemma~\ref{L:CollPair} implies that the pair is
$\II$-collapsing. 
\end{proof}

\begin{claim}
\label{C:2}
A special pair has a unique type, except, for $\qqp \le \bb < \aa < \nn$,

- the pairs $\Qp(..., \bbp, \aa, \nnp)$, which have type both~$\tIVp_\aa$ and~$\tVIp_\bb$,

- the pairs $\Qm(\qq, \bbp, \aa, \nnp)$, which have type both~$\tIIm_\aa$ and~$\tVIm_\bb$.
\end{claim}

\begin{proof}
A positive pair cannot coincide with a negative one, so we can consider positive and negative pairs separately. Next,
for each type~$\tau_\aa$, the value of~$\aa$ can be recovered from one of the parameters~$\nu_\kk$ of the pair,
hence it is impossible that a pair of type~$\tau_\aa$ be of type~$\tau_\bb$ for some $\bb \not=\aa$. So it remains
to check that, for each pair $(\tau_\aa, \tau'_\bb)$ of distinct types (with the same sign), it is impossible that a
special pair be simultaneously of type~$\tau_\aa$ and~$\tau'_\bb$---up to the exceptions mentioned in the
claim. This is done in the two arrays below. As should be clear, ``$\nu_2 \gs \qq/\nu_2 \lse \qq$'' means that a pair
cannot be of the considered two types~$\tau, \tau'$ simultaneously because, for being of type~$\tau$, we must have
$\nu_2 > \qq$ whereas, for being of type~$\tau'$, we must have $\nu_2 \le \qq$.

\begin{center}
\begin{tabular}{c|c|c|c|c|c}
\vline width0em height1em depth0.5em
$\cap$
&$\tIIp_\bb$
&$\tIIIp_\bb$
&$\tIVp_\bb$
&$\tVp_\bb$
&$\tVIp_\bb$\\
\hline
\vline width0em height1em depth0.5em
$\tIp_\aa$
&$\nu_3{=}\qq/\nu_3\gs\qq$
&$\nu_3{=}\qq/\nu_3\ls\qq$
&$\nu_3{=}\qq/\nu_3\gs\qq$
&$\nu_3{=}\qq/\nu_3\gs\qq$
&$\nu_3{=}\qq/\nu_3\gs\qq$\\
\hline
\vline width0em height1em depth0.5em
$\tIIp_\aa$
&&$\nu_2{=}\qq/\nu_2\ls\qq$
&$\nu_2{=}\qq/\nu_2{\not=}\qq$
&$\nu_2{=}\qq/\nu_2\lse\ppp$
&$\nu_2{=}\qq/\nu_2\gs\qq$\\
\hline
\vline width0em height1em depth0.5em
$\tIIIp_\aa$
&&&$\nu_3\ls\qq/\nu_3\gs\qq$
&$\nu_3\ls\qq/\nu_3\gs\qq$
&$\nu_2\ls\qq/\nu_2\gs\qq$\\
\hline
\vline width0em height1em depth0.5em
$\tIVp_\aa$
&&&&$\nu_2\gse\pppp/\nu_2\lse\ppp$
&possible\\
\hline
\vline width0em height1em depth0.5em
$\tVp_\aa$
&&&&&$\nu_4\lse\nn/\nu_4\gs\nn$\\
\end{tabular}

\begin{tabular}{c|c|c|c|c}
\vline width0em height1em depth0.5em
$\cap$
&$\tIIm_\bb$
&$\tIIIm_\bb$
&$\tVm_\bb$
&$\tVIm_\bb$\\
\hline
\vline width0em height1em depth0.5em
$\tIm_\aa$
&$\nu_2\lse\qq/\nu_2\gs\qq$
&$\nu_4\lse\nn/\nu_4{=}\nnp$
&$\nu_2\gse\pppp/\nu_2\lse\ppp$
&$\nu_4{=}\qq/\nu_4{=}\nnp$\\
\hline
\vline width0em height1em depth0.5em
$\tIIm_\aa$
&&$\nu_1{=}\qq/\nu_1\ls\qq$
&$\nu_1{=}\qq/\nu_1\lse\pp$
&possible\\
\hline
\vline width0em height1em depth0.5em
$\tIIIm_\aa$
&&&$\nu_2\gse\pppp/\nu_2\lse\ppp$
&$\nu_2\lse\qq/\nu_2\gs\qq$\\
\hline
\vline width0em height1em depth0.5em
$\tVm_\aa$
&&&&$\nu_4\lse\nn/\nu_4\gs\nn$\\
\end{tabular}
\end{center}

So all cases have been considered.
\end{proof}

We shall now exploit the differences of covering and co-covering
between~$\TT$ and~$\TTT$. First, we use the fact that $\ppp$
to~$\qqm$ are not covered by~$\qq$ in~$\TT$, whereas they are
in~$\TTT$, and, symmetrically, the fact that $\qq$ co-covers $\qqpp$
to~$\nn$ in~$\TT$, whereas it does not in~$\TTT$.

\begin{claim}
\label{C:3}
Every path from~$\TT$ to~$\TTT$ 
\vspace{-0.2em}
\begin{tabbing}
\quad
\=\qquad 
\=- contains a pair of type~$\tIpm_\aa$ \hspace{0.1em}
\=for each~$\aa$ in~$[\pppp, \qq]$, \\
\>and 
\>- contains a pair of type~$\tIIpm_\aa$ 
\>for each~$\aa$ in~$[\qqp, \nnm]$.
\end{tabbing}
\end{claim}

\begin{proof}
Assume $\aa \in [\pppp, \qq]$. Then we have $\ppp \le \aam \le 2\pp$, and, as can be read on
Figure~\ref{F:Key}, we have
$$\aam \not\cov\TT \qq, \quad
\aam \cov\TTT \qq, \quad\text{and}\quad
\aa \ccove\TTT \qq.$$
Applying Lemma~\ref{L:PPrinciple} with $\ii = \aam$ and $\jj = \qq$ shows that every path from~$\TT$
to~$\TTT$ contains a pair $\Qp(..., \aa, \qq, ...)$, of type~$\tIp_\aa$, or a pair $\Qm(..., \aa, ..., \qq)$, of
type~$\tIm_\aa$.

Assume now $\aa \in [\qqp, \nn-1]$. Then we have $\qq+2 \le \aap \le \nn$, and, as can be read on
Figure~\ref{F:Key}, we have
$$\qq \ccov\TT \aap, \quad
\qq \cove\TT \aa, \quad\text{and}\quad
\qq \not\ccov\TTT \aap.$$
Applying Lemma~\ref{L:PPrincipleBis} with $\ii = \qq$ and $\jj = \aap$ shows that every path from~$\TT$
to~$\TTT$ contains a pair $\Qp(..., \qq, \aa, ...)$, of type~$\tIIp_\aa$, or a pair $\Qm(\qq, ..., \aa, ...)$, of
type~$\tIIm_\aa$.
\end{proof}

\begin{rema}
The second argument above also applies to $\aa = \qq$, but then it possibly leads to a pair $\Qp(..., \qq, \qq, ...)$
that would be of type~$\tIp_\qq$ and may have been already considered in the first argument.
\end{rema}

We shall now exploit the fact that $\ppp$ is not covered by any label
from~$\qqp$ to~$\nnm$ in~$\TT$, and is covered by these labels
in~$\TT$.

\begin{claim}
\label{C:4}
Every path from~$\TT$ to~$\TTT$
\vspace{-0.2em}
\begin{tabbing}
\qquad 
\=- contains a pair of type~$\tIIIp_\aa$ \hspace{0.1em}
\=for each~$\aa$ in~$[\pppp, \qqm]$, \\
or 
\>- contains a pair of type~$\tIV_\bb$ 
\>for each~$\bb$ in~$[\qqp, \nnm]$.
\end{tabbing}
\end{claim}

\begin{proof}
For $\pp = 1$, the result is vacuously true, and we assume $\pp \ge 2$. Let $(\TT_0, \Ldots, \TT_\ell)$ be any path
from~$\TT$ to~$\TTT$. Assume that, for some~$\bb$ in~$[\qqp, \nnm]$, there exists no pair of
type~$\tIV_\bb$ in $(\TT_0, \Ldots, \TT_\ell)$. As we read on Figure~\ref{F:Key} the relations
$$\ppp \not\cov\TT \bb \quad\text{and}\quad \ppp\cov\TTT \bb,$$
applying Lemma~\ref{L:KeyLemma1} with $\ii = \ppp$ and $\jj = \bb$ shows that there must exist~$\rr$
such that $(\TT_\rr, \TT_\rrp)$ is $\Qp(\hh, \gg, \bb, ...)$ for some~$\gg, \hh$ satisfying $\hh \le\ppp$ and $\gg
\ge\pppp$. The hypothesis that this pair is \emph{not} of type~$\tIV_\bb$ implies $\gg = \qq$. As
$\pp \ge 2$ holds, we have $\hh \le \ppp < \qqm$ and, therefore, Lemma~\ref{L:NameCov}$(i)$
implies $\hh \cov{\TT_\rr} \qqm$. Let $\aa$ be any element of~$[\pppp, \qqm]$. By Lemma~\ref{L:CovInterval}, 
$\hh \cov{\TT_\rr} \qqm$ implies
\linebreak
 $\aam \cov{\TT_\rr} \qqm$. On the other hand, we see on
Figure~\ref{F:Key} that $\aa$ co-covers~$\qq$ in~$\TT$. Applying Lemma~\ref{L:PPrincipleBis} with $\ii
= \aa$ and $\jj = \qq$ shows that $(\TT_0, \Ldots, \TT_\rr)$ contains a pair $\Qp(..., \aa, \ls\qq, \gse\qq)$, of
type~$\tIIIp_\aa$.
\end{proof}

The next result uses the fact that each label~$\aa$ between~$2$
and~$\pp$ is not covered by~$\nnp-\aa$ in~$\TT$, whereas it is
in~$\TTT$. The possible interference of the label~$\nnp$ makes the result
slightly more complicated---as was already the case in the proof of
Lemma~\ref{L:Tricomb}. This step is the most delicate one, as it requires
the full power of Lemma~\ref{L:PPPrinciple} and not only
Lemma~\ref{L:KeyLemma1} or Lemma~\ref{L:PPrinciple}.

\begin{claim}
\label{C:5}
For each~$\aa$ in~$[\qqp, \nnm]$, every path from~$\TT$ to~$\TTT$
\vspace{-0.1em}
\begin{tabbing}
\qquad 
\=- contains a pair of type~$\tVpm_\aa$, \\
or 
\>- contains a pair of type~$\tVIp_\aa$ and a pair of type~$\tVIm_\aa$.
\end{tabbing}
\end{claim}

\begin{proof}
Let $(\TT_0, \Ldots, \TT_\ell)$ be a path from~$\TT$ to~$\TTT$, and let~$\aa$ be an
element of~$[\qqp, \nnm]$. Put $\aa' = \nnp - \aa$. Then we have $2 \le \aa' \le \pp$, and we read on
Figure~\ref{F:Key} the relations
$$\aa' \not\cov\TT \aa, \quad
\aa' \cov\TTT \aa, \quad\text{and}\quad
\aa'+1 \ccove\TTT \aa.$$
Applying Lemma~\ref{L:PPPrinciple} with $\ii = \aa'$, $\jj = \aa$, and $\kk = \ppp$ shows that there
exists~$\rr$ such that $(\TT_\rr, \TT_\rrp)$ is $\Qp(\lse\aa', \aa'\ls...\lse\ppp, \aa, ...)$ or
$\Qm(\lse\aa', \aa'\ls...\lse\ppp, \gs\pp, \aa)$. The latter pair is of type~$\tVm_\aa$. As for the former one, two
cases are possible: if $\Name4{\TT_\rr}{\TT_\rrp} \le \nn$ holds, then
$(\TT_\rr, \TT_\rrp)$ is of type~$\tVp_\aa$, else we necessarily have
$\Name4{\TT_\rr}{\TT_\rrp} = \nn$. Then $(\TT_\rr, \TT_\rrp)$ is  a
pair $\Qp(..., ..., \aa, \nnp)$, in which case, by
Lemma~\ref{L:NameCov}$(iv)$, we have $\aap \ccove{\TT_\rr} \nnp$,
and even $\aap
\ccov{\TT_\rr} \nnp$ as $\aa < \nn$ is assumed. Now, we read on Figure~\ref{F:Key} that $\aap$
co-covers~$\nnp$ neither in~$\TT$ nor in~$\TTT$. Applying Lemma~\ref{L:KeyLemma2} with $\ii = \aap$
and $\jj = \nnp$ guarantees that $(\TT_0, \Ldots, \TT_\rr)$ contains at least a pair $\Qm(..., \aap, ..., \nnp)$,
of type~$\tVIm_\aa$, and that $(\TT_\rr, \Ldots, \TT_\ell)$ contains at least a pair $\Qp(..., \aap, ..., \nnp)$,
of type~$\tVIp_\aa$.
\end{proof}

The last claim of the series will be used to cope with the possible
interference between types~$\tIIm$ and~$\tVIm$.

\begin{claim}
\label{C:6}
Assume that $\TTs$ is a size~$\nn$ tree and $\qq$ co-covers~$\nnp$ in~$\TTs$. Then every path
from~$\TT$ to~$\TTs$ contains

- a pair of type~$\tIIIpm_\aa$ for each~$\aa$ in~$[\pppp, \qq]$.
\end{claim}

\begin{proof}
Let $\TTb = \coll{[\qqp, \nn]}(\TT)$ and $\TTsb = \coll{[\qqp, \nn]}(\TTs)$. Then we have $\TTb =
\Sp{01^\pp}$, and, by Lemma~\ref{L:CollCov}, the hypothesis that $\qq$ co-covers~$\nnp$ in~$\TTs$ implies
that $\qq$ co-covers~$\nnp$ in~$\TTsb$. Let~$\aa$ be any element of~$[\pppp, \qq]$. By
Lemma~\ref{L:CovShift}, $\aam$ cannot be covered by~$\nn$ in~$\TTsb$. On the other hand, $\aam$ is
covered by~$\nn$ in~$\TT$, and, moreover, $\aa$ co-covers~$\qq$ in~$\TT$. Applying
Lemma~\ref{L:PPrinciple} to~$(\TTsb, \TTb)$ with $\ii = \aa$ and $\jj = \qq$ shows that every path
from~$\TTb$ to~$\TTsb$ contains a pair $\Qm(..., \aa, \qq, \nnp)$ or $\Qp(..., \aa, ..., \qq)$. It follows that
every path from~$\TT$ to~$\TTs$ contains a pair that projects to a pair of the
previous form when $[\qqp, \nn]$ is collapsed. 

Now, Lemma~\ref{L:CollPair} shows that
$\Qm(\aa', \bb', \cc', \dd')$ projects to $\Qm(..., \aa, \qq, \nnp)$ if and only if we have $\bb' = \aa$, $\sup([\bb',
\cc'] \setminus [\qqp, \nn]) = \qq$, hence $\cc' \ge \qq$, and $\sup([\cc'+1, \dd'] \setminus [\qqp,
\nn]) = \nnp$, hence $\dd' = \nnp$, implying that $\Qm(\aa', \bb', \cc', \dd')$ is special of type~$\tIIIm_\aa$.

Similarly,  $\Qp(\aa', \bb', \cc', \dd')$ projects to $\Qp(..., \aa, ..., \qq)$ for $\bb' = \aa$, $\cc' =
\sup([\bb', \cc'] \setminus [\qqp, \nn]) < \qq$, hence $\cc' < \qq$, and $\sup([\cc'+1, \dd'] \setminus
[\qqp, \nn]) = \qq$, hence $\dd' \le \nn$, so that $\Qp(\aa', \bb', \cc', \dd')$ is special of type~$\tIIIp_\aa$.
\end{proof}

\subsection{Proof of Lemma~\ref{L:Multicomb}}

We are now ready to prove Lemma~\ref{L:Multicomb}. The argument is similar 
to the one used for proving Lemma~\ref{L:Tricomb}, but we have to be
more careful because of the possible interferences between special pairs of
type~$\tVIpm$ and special pairs of other types. It may be noted that such
problems never occur when $\add{}\nn$ remains~$1^\pp01^\ppp$
throughout the considered path from~$\TT$ to~$\TTT$: proving the
result for the trees $\coll{\{\nnp\}}(\TT)$ and
$\coll{\{\nnp\}}(\TTT)$ would be easier.

\begin{proof}[Proof of Lemma~\ref{L:Multicomb}]
Let $(\TT_0, \Ldots, \TT_\ell)$ be any path from~$\TT$ to~$\TTT$. We have to show that this path
contains at least $4\pp-3$ special pairs. The latter can correspond to several combinations of types, and we
consider three cases.

\medskip
\noindent{\bf Case 1:} $(\TT_0, \Ldots, \TT_\ell)$ contains a pair of type~$\tIIIp_\aa$ for each~$\aa$
in~$[\pppp, \qqm]$.

Then $(\TT_0, \Ldots, \TT_\ell)$ contains at least

\begin{tabbing}
\hspace{3em}
\=\quad $\pp$ \quad
\=pairs of type~$\tIpm$\hspace{5em}
\=by Claim~\ref{C:3},\\
\>$\ppm$ 
\>pairs of type~$\tIIpm$
\>by Claim~\ref{C:3},\\
\>$\ppm$ 
\>pairs of type~$\tIIIp$
\>by hypothesis,\\
\>$\ppm$ 
\>pairs of type~$\tVpm$ or~$\tVIp$
\>by Claim~\ref{C:5},
\end{tabbing}
hence at least $4\pp-3$ special pairs, which are $\II$-collapsing by Claim~\ref{C:1}, and are pairwise distinct by
Claim~\ref{C:2}, since we appeal to no pair of type~$\tIVp$ (which could interfer with~$\tVIp$) or~$\tVIm$ (which
could interfer with~$\tIIm$).

\medskip
\noindent{\bf Case 2:} There exists~$\aa$ in~$[\pppp, \qqm]$ such that $(\TT_0, \Ldots, \TT_\ell)$ contains no
pair of type~$\tIIIp_\aa$, and there is no~$\rr$ such that $\qq$ co-covers~$\nnp$ in~$\TT_\rr$.

Then Claim~\ref{C:2} implies that $(\TT_0, \Ldots, \TT_\ell)$ contains no special pair that is simultaneously of
types~$\tIIm$ and~$\tVIm$. Indeed, if a pair is both of types~$\tIIm_\aa$ and $\tVIm_\bb$, it has the form
$\Qm(\qq, ..., ..., \nnp)$, and, therefore, by Lemma~\ref{L:NameCov}$(i)$, $\qq$ co-covers~$\nnp$ in the two trees
of that pair. Then $(\TT_0, \Ldots, \TT_\ell)$ contains

\begin{tabbing}
\hspace{3em}
\=\quad $\pp$ \quad
\=pairs of type~$\tIpm$\hspace{5em}
\=by Claim~\ref{C:3},\\
\>$\ppm$ 
\>pairs of type~$\tIIpm$
\>by Claim~\ref{C:3},\\
\>$\ppm$ 
\>pairs of type~$\tIVp$
\>by Claim~\ref{C:4} and the hypothesis,\\
\>$\ppm$ 
\>pairs of type~$\tVpm$ or~$\tVIm$
\>by Claim~\ref{C:5},
\end{tabbing}
and we have again $4\pp-3$ special pairs, which are $\II$-collapsing by Claim~\ref{C:1}, and are pairwise distinct,
since we appeal to no pair of type~$\tVIp$ (which could interfer with~$\tIVp$), and interferences between
type~$\tIIm$ and~$\tVIm$ are discarded by our hypotheses.

\medskip
\noindent{\bf Case 3:} There exists~$\aa$ in~$[\pppp, \qqm]$ such that
$(\TT_0, \Ldots, \TT_\ell)$ contains no pair of type~$\tIIIp_\aa$, and
there exists~$\rr$ such that $\qq$ co-covers~$\nnp$ in~$\TT_\rr$.

Then $(\TT_0, \Ldots, \TT_\ell)$ contains

\begin{tabbing}
\hspace{3em}
\=\quad $\pp$ \quad
\=pairs of type~$\tIpm$\hspace{5em}
\=by Claim~\ref{C:3},\\
\>$\ppm$ 
\>pairs of type~$\tIIpm$
\>by Claim~\ref{C:3},\\
\>$\ppm$ 
\>pairs of type~$\tIIIpm$
\>by Claim~\ref{C:6},\\
\>$\ppm$ 
\>pairs of type~$\tIVp$
\>by Claim~\ref{C:5} and the hypothesis.
\end{tabbing}
We still have found $4\pp-3$ special pairs, all $\II$-collapsing by Claim~\ref{C:1}, and pairwise distinct
by Claim~\ref{C:2}, since we appeal now to no pair of type~$\tVIpm$.
\end{proof}

So the proofs of Lemma~\ref{L:Multicomb} and, therefore, of
Proposition~\ref{P:Multicomb}, are complete, yielding the exepected lower
bound $\dd(\nn) \ge\nobreak 2\nn - O(\sqrt\nn)$ on the diameter of the
$\nn$th associahedron.

\subsection{Going further}

Proving the conjectured value $\dd(\nn) = 2\nn - 6$ for $\nn >
10$ using the above methods seems feasible, but is likely to require a
more intricate argument. Experiments easily suggest families of trees that
should achieve the maximal distance, namely symmetric zigzag-trees with
small combs attached at each end, on the shape of the
trees~$\TTT_{\mm,3}$ of Proposition~\ref{P:Multicomb}.

\begin{conj}
\label{C:Final}
Define
$$\TT_\nn = 
\begin{cases}
\Sp{111(01)^{\pp-3}00}\\
\Sp{111(01)^{\pp-3}000}
\end{cases}
\TTT_\nn = 
\begin{cases}
\Sp{000(10)^{\pp-3}11}
&\mbox{for $\nn = 2\pp-1$},\\
\Sp{000(10)^{\pp-3}111}
&\mbox{for $\nn = 2\pp$},
\end{cases}
$$
Then one has $\dist(\TT_\nn, \TTT_\nn) = 2\nn-6$ for $\nn \ge 11$.
\end{conj}

The problem for establishing Conjecture~\ref{C:Final} is that counting
$\II$-collapsing pairs cannot suffice here: various solutions exist for
projecting~$(\TT_\nn, \TTT_\nn)$ onto~$(\TT_{\nn-2}, \TTT_{\nn-2})$
by collapsing two labels, but the minimal number of collapsing pairs is
then~$3$, and not~$4$, as would be needed to conclude. On the other
hand, the highly symmetric shape of the  trees~$\TT_\nn$
and~$\TTT_\nn$ allows for other arguments that will not be developped
here. 

\section*{Acknowledgment}

The author thanks Shalom Eliahou for introducing him to the fascinating
problem addressed in this paper.

\end{document}